\documentclass[11pt]{amsart}
\usepackage[dvips]{graphicx}
\usepackage{amsmath,graphics,xypic}
\usepackage{amsfonts,amssymb}
\theoremstyle{plain}
\newtheorem{theorem*}{Theorem}
\newtheorem*{lemma*}{Lemma}
\newtheorem{corollary*}{Corollary}
\newtheorem*{proposition*}{Proposition}
\newtheorem{conjecture*}{Conjecture}
\newtheorem{theorem}{Theorem}[section]
\newtheorem{lemma}[theorem]{Lemma}
\newtheorem{corollary}[theorem]{Corollary}
\newtheorem{proposition}[theorem]{Proposition}

\theoremstyle{remark}

\newtheorem*{remark}{Remark}
\newtheorem*{definition}{Definition}

\theoremstyle{definition}

   \def\Z{\Bbb{Z}}  \def\C{\Bbb{C}}
    
  \def\g{\gamma}  \def\bp{\begin{pmatrix}}
\def\sm{\setminus} \def\ep{\end{pmatrix}} \def\bn{\begin{enumerate}} 
   \def\en{\end{enumerate}}
\def\ba{\begin{array}} \def\ea{\end{array}}  
 \def\S{\Sigma}

\def\ker{\mbox{Ker}}\def\be{\begin{equation}} \def\ee{\end{equation}}

 \def\aut{\mbox{Aut}}

     \def\fr12{\frac{1}{2}} \def\z12{\Z[\fr12]}

\def\T{\mathcal{T}}
\def\tkt{\T_K(t)}

\begin{document}

\title{A survey of twisted Alexander polynomials}
\author{Stefan Friedl}
\address{University of Warwick, Coventry, UK}
\email{s.k.friedl@warwick.ac.uk}
\author{Stefano Vidussi}
\address{Department of Mathematics, University of California,
Riverside, CA 92521, USA} \email{svidussi@math.ucr.edu}

\date{\today}
\begin{abstract}
We give  a short introduction to  the theory of twisted Alexander polynomials  of a $3$--manifold associated to a representation of its fundamental group. We summarize their formal properties and we explain their relationship to twisted Reidemeister torsion. We then give a survey of the many applications of twisted invariants to the study of topological problems. We conclude with a short summary of the theory of higher order Alexander polynomials.
\end{abstract}
\maketitle

\section{Introduction}

In 1928 Alexander introduced a polynomial invariant for knots and links which quickly got referred  to as the Alexander polynomial.
His definition was later recast in terms of Reidemeister torsion by Milnor \cite{Mi62}
and it was extended by Turaev \cite{Tu75,Tu86} to an invariant of 3-manifolds.
More precisely, to a 3-manifold with empty or toroidal boundary $N$ we can associate its Alexander polynomial $\Delta_N$
which lies in the group ring $\Z[H]$, where $H$ is the maximal abelian quotient of $H_1(N;\Z)$.
The Alexander polynomial of knots, links and 3-manifolds in general is closely related to the topology properties of the underlying space.
For example it is known to contain information on the knot genus \cite{Se35}, knot concordance (\cite{FM66}), fiberedness and symmetries.

The Alexander polynomial carries only metabelian information on the fundamental group.
This limitation explain why in all the above  cases the Alexander polynomial carries partial, but not complete information.
The idea behind twisted invariants is to associate a polynomial invariant to a 3-manifold \emph{together}
with a choice of a representation of its fundamental group.
This approach makes it possible to
extract more powerful topological information.

The twisted Alexander polynomial for a knot $K \subset S^3$
was first  introduced by Xiao--Song Lin in 1990 (cf. \cite{Lin01}).
Whereas Lin's original definition used `regular Seifert surfaces' of knots,
later extensions to links and 3--manifolds either
generalized the Reidemeister--Milnor--Turaev torsion (cf. \cite{Wa94,Ki96,KL99a,FK06}) or
generalized the homological definition of the Alexander polynomial (cf. \cite{JW93,KL99a,Ch03,FK06,HKL08}).

In most cases the setup for twisted invariants is as follows:
Let $N$ be a 3--manifold with empty or toroidal boundary,  $\psi:\pi_1(N)\to F$ an epimorphism onto a free abelian group $F$
and $\gamma:\pi_1(N)\to \mbox{GL}(k,R)$  a representation with $R$ a domain.
In that case one can define the  twisted Reidemeister torsion $\tau(N,\gamma\otimes \psi)$, an invariant which in general lives in the quotient field of
the group ring $R[F]$.
If $R$ is furthermore a Noetherian unique factorization domain (e.g. $R=\mathbb{Z}$ or $R$ a field),
then  the twisted Alexander polynomials $\Delta_{N,i}^{\gamma\otimes \psi}\in R[F]$ is defined to be
the order of the twisted Alexander module $H_i(N;R[F]^k)$.

These two invariants are closely related, for example in the case that $\mbox{rank}(F)\geq 2$ we will see  that
\[ \tau(N,\gamma\otimes \psi)=\Delta_{N,1}^{\gamma\otimes \psi}\in R[F]. \]
 In fact in many papers the twisted Reidemeister torsion $\tau(N,\gamma\otimes \psi)$
is referred to as the twisted Alexander polynomial (cf. e.g. \cite{Wa94}).

The most important raison d'\^etre of these invariants lies in the fact that they contain deep information on the underlying topology
while at the same time being, as we will see, very computable invariants.

We now give a short outline of the paper.
In Section \ref{section:definition} we define twisted Reidemeister torsion and twisted Alexander polynomials of 3-manifolds,
and we show how to calculate these invariants.
In Section \ref{section:basics} we discuss basic properties of twisted invariants,
in particular we discuss the  relationship between twisted Reidemeister torsion and twisted Alexander polynomials
and we discuss the effect of Poincar\'e duality on twisted invariants.
Section \ref{section:dist} contains applications to distinguishing knots and links using twisted invariants.
In Section \ref{section:conc} we outline the results of Kirk and Livingston
regarding the behavior of twisted invariants under knot concordance and we extend the results
to the study of doubly slice knots and ribbon knots.
In Section \ref{section:fibgenus} we show that twisted invariants
give lower bounds on the knot genus and the Thurston norm, and we show that they give obstructions to
the fiberedness of 3-manifolds.
Sections \ref{section:twodim}, \ref{section:misc} and \ref{section:general}
contains a discussion of the many generalizations and further applications of twisted invariants.
In Section \ref{section:ho} we give an overview of the closely related theory of higher-order Alexander polynomials,
this theory was initiated by Cochran and Harvey.
Finally in Section \ref{section:question} we provide a list of open questions and problems.

\noindent \textbf{Conventions and Notation.}
Unless we say otherwise we adopt the following conventions:\\
(1) rings are commutative domains with unit element,\\
(2) 3--manifolds are  compact, connected and orientable,\\
(3) homology is taken with integral coefficients,\\
(4) groups are finitely generated.\\
We also use the following notation:
Given a ring $R$ we denote by $Q(R)$ its quotient field
and given a link $L\subset S^3$ we denote by $\nu L$ a (open) tubular neighborhood of $L$ in $S^3$.

\noindent \textbf{Remark.} For space reasons we unfortunately have to exclude from our exposition several important aspects of the subject. Among the most relevant omissions, we mention Turaev's torsion function, and the relation between torsion invariants on the one hand and Seiberg--Witten theory and Heegaard--Floer homology on the other.
Turaev's torsion function is defined using Reidemeister torsion corresponding to one--dimensional abelian representations.
This theory and its connection to Seiberg--Witten invariants, first unveiled by Meng and Taubes in \cite{MT96} (cf. also \cite{Do99}), is treated beautifully in Turaev's original papers
\cite{Tu97,Tu98} and in Turaev's books \cite{Tu01,Tu02a}. We also refer to \cite{OS04} for the relation of Turaev's torsion function to Heegaard--Floer homology.

\noindent \textbf{Remark.}
Almost all the results of this survey paper appeared already in previous paper.
We hope that we correctly  stated  the results of the many authors who worked on twisted Alexander polynomials. For the definite statements we nonetheless refer to the original papers.
The only new results are some theorems in Section \ref{section:slice} on knot and link concordance and
Theorem \ref{thm:fibob} on Reidemeister torsion of fibered manifolds.



%



\noindent \textbf{Acknowledgments.}
The first author would like to thank Markus Banagl and Denis Vogel for organizing the workshop
`The Mathematics of Knots: Theory and Application' in Heidelberg in December 2008.
The authors also would like to thank Takahiro Kitayama, Taehee Kim, Chuck Livingston, Takayuki Morifuji, Andrew Ranicki and
Dan Silver for many helpful comments and remarks. Finally we wish to thank the referee for many helpful comments and remarks.

\section{Definition and basic properties} \label{section:definition}

\subsection{Twisted Reidemeister torsion}\label{section:twitorsion}\label{section:twitau}

Let $N$ be a 3--manifold with empty or toroidal boundary, $F$ a torsion--free abelian group
and $\alpha:\pi_1(N)\to \mbox{GL}(k,R[F])$  a representation. Recall that we denote by $Q(R[F])$ the quotient field of $R[F]$.

We endow $N$ with a finite CW--structure.
We denote the universal cover of $N$ by $\tilde{N}$.
Recall that there exists a canonical left $\pi_1(N)$--action on the universal cover $\tilde{N}$ given by deck transformations. We
consider the cellular chain complex $C_*(\tilde{N})$ as a right $\mathbb{Z}[\pi_1(N)]$-module by defining $\sigma \cdot
g\mathrel{\mathop:}= g^{-1}\sigma$ for a chain $\sigma$. The representation $\alpha$ induces a representation $\alpha:\pi_1(N)\to \mbox{GL}(k,R[F])\to \mbox{GL}(k,Q(R[F]))$
which gives rise to  a left action of $\pi_1(N)$ on $Q(R[F])^k$. We can therefore consider
the  $Q(R[F])$--complex
\[ C_*(\tilde{N})\otimes_{\mathbb{Z}[\pi_1(N)]}Q(R[F])^k.\]
We now endow the free $\mathbb{Z}[\pi_1(N)]$--modules $C_*(\tilde{N})$
with a basis by picking lifts of the cells of $N$ to $\tilde{N}$.
Together with the canonical basis for $Q(R[F])^k$ we can now view  the $Q(R[F])$--complex $C_*(\tilde{N})\otimes_{\mathbb{Z}[\pi_1(N)]}Q(R[F])^k$
as a complex of based $Q(R[F])$--modules.

If this complex is not acyclic, then we define $\tau(N,\alpha)=0$.
Otherwise we denote by $\tau(N,\alpha)\in Q(R[F])\setminus \{0\}$ the  Reidemeister torsion of this based $Q(R[F])$--complex.
We will not recall the definition of Reidemeister torsion, referring instead to the many excellent expositions,
 e.g. \cite{Mi66}, \cite{Tu01} and \cite{Nic03}. However, in the next section we will
present a method for computing explicitly the twisted Reidemeister torsion of a 3--manifold.

It follows from Chapman's theorem \cite{Chp74} and from standard arguments (cf. the above literature) that
up to multiplication by an element in
\[     \{ \pm\det(\alpha(g)) \, |\, g\in \pi_1(N)\}\]
the Reidemeister torsion $\tau(N,\alpha)$ is well--defined, i.e. up to that indeterminacy $\tau(N,\alpha)$ is independent of the choice of underlying CW--structure
and the choice of the lifts of the cells. In the following, given $w\in Q(R[F])$ we write
\[ \tau(N,\alpha) \doteq w \]
if there exists a representative of $\tau(N,\alpha)$ which equals $w$.

Note that if $\gamma:\pi_1(N)\to \mbox{GL}(k,R)$ is a representation and $\psi:\pi_1(N)\to F$ a  homomorphism to a free abelian group, then we get a tensor representation
\[ \begin{array}{rcl} \gamma\otimes \psi: \pi_1(N) &\to & \mbox{GL}(k,R[F]) \\
g&\mapsto & \gamma(g)\cdot \psi(g)\end{array} \]
and the corresponding Reidemeister torsion $\tau(N,\gamma \otimes \psi)$. Except for parts of Section \ref{section:slice}
we will always consider the twisted Reidemeister torsion corresponding to such a tensor representation.
In that case, specializing the previous formula, $\tau(N,\gamma \otimes \psi)$  is well--defined up to multiplication by an element  in
\[     \{\pm \det(\gamma(g)) f\, |\, g\in \pi_1(N), f\in F\}.\]
In particular, if $\gamma:\pi_1(N)\to \mbox{SL}(k,R)$ is a  representation to a special linear groups,
then $\tau(N,\gamma \otimes \psi)\in Q(R[F])$ is well--defined up to multiplication
by an element in $\pm F$.
Furthermore, if $k$ is even, then $\tau(N,\gamma\otimes \psi)\in Q(R[F])$ is in fact well--defined up to multiplication
by an element in $F$ (cf. e.g. \cite{GKM05}).

\noindent Finally we adopt the following notation:
\begin{enumerate}
\item Given a homomorphism $\gamma:\pi\to G$ to a finite group $G$
we get an induced representation $\pi\to \mbox{Aut}(\mathbb{Z}[G])\cong \mbox{GL}(|G|,\mathbb{Z})$
given by left multiplication. In our notation we will not distinguish between a homomorphism to a finite group and the corresponding representation over $\mathbb{Z}$.
\item If $N$ is the exterior of a link $L\subset S^3$, $\psi:\pi_1(S^3\setminus \nu L)\to F$ the abelianization
and $\gamma:\pi\to \mbox{GL}(k,R)$ a representation, then we  write
$\tau(L,\gamma)$ for $\tau(S^3\setminus \nu L,\gamma\otimes \psi)$.
\end{enumerate}

\subsection{Computation of twisted Reidemeister torsion}\label{section:comptau}

Let $N$ be a 3--manifold with empty or toroidal boundary,  $\psi:\pi_1(N)\to F$ a non--trivial homomorphism to a free abelian group $F$
and $\gamma:\pi_1(N)\to \mbox{GL}(k,R)$  a representation.
In this section we will give an algorithm for computing $\tau(N,\gamma \otimes \psi)$ which is based on ideas of Turaev
(cf. in particular \cite[Theorem~2.2]{Tu01}).

We will first consider the case  that $N$ is closed.  We write $\pi=\pi_1(N)$.
We endow $N$ with  a CW--structure with one 0--cell, $n$ 1--cells, $n$ 2--cells and one $3$--cell.
It is well--known that such a CW--structure exists (cf. e.g. \cite[Theorem 5.1]{McM02} or \cite[Proof~of~Theorem~6.1]{FK06}).
Using this CW--structure we have the cellular chain complex
\[
0 \to C_3(\tilde{N}) \xrightarrow{\partial_3} C_2(\tilde{N})
\xrightarrow{\partial_2} C_1(\tilde{N}) \xrightarrow{\partial_1}
C_0(\tilde{N}) \to 0
\]
where $C_i(\tilde{N})\cong \mathbb{Z}[\pi]$ for $i=0,3$ and
$C_i(\tilde{N})\cong \mathbb{Z}[\pi]^n$ for $i=1,2$. Let $A_i,
i=1,2,3$  be the matrices over $\mathbb{Z}[\pi]$ corresponding to
the boundary maps $\partial_i:C_i\to C_{i-1}$  with respect to the
bases given by the lifts of the cells of $N$ to $\tilde{N}$. We can
arrange the lifts such that
 \[ \begin{array}{rcl} A_3 &=&
(1-g_1, 1-g_2, \ldots, 1-g_n)^t,\\
A_1 &=& (1-h_1, 1-h_2, \ldots, 1-h_n) \end{array} \]
with $g_1,\dots,g_n,h_1,\dots,h_n \in \pi$.
Note that $\{g_1,\dots,g_n\}$ and $\{h_1,\dots,h_n\}$
are generating sets for $\pi$ since $N$ is a closed 3--manifold. Since $\psi$ is non--trivial
there exist $r,s$ such that $\psi(g_r)\ne 0$ and $\psi(h_s)\ne 0$. Let $B_3$ be the $r$--th row of
$A_3$. Let $B_2$ be the result of deleting the $r$-th column and the $s$--th row from $A_2$. Let
$B_1$ be the $s$--th column of $A_1$.

Given a $p\times q$ matrix $B = (b_{rs})$  with entries in $\mathbb{Z}[\pi]$ we write $b_{rs}=\sum
b_{rs}^gg$ for $b_{rs}^g\in \mathbb{Z}, g\in \pi$. We then define $(\gamma \otimes \psi)(B)$ to be the $p\times
q$ matrix with entries
\[ \sum b_{rs}^g (\gamma \otimes \psi)(g)=\sum b_{rs}^g
\gamma(g)\cdot {\psi(g)}\in R[F]. \]
Since each such entry  is a $k\times k$ matrix with entries in $R[F]$ we can think of  $(\gamma \otimes \psi)(B)$ as a $pk\times qk$ matrix with entries in $R[F]$.

Now note that
\[ \det((\gamma \otimes \psi)(B_3))=\det(\mbox{id}-\gamma(g_r)\cdot {\psi(g_r)})\ne 0 \]
 since $\psi(g_r)\ne 0$. Similarly $\det((\gamma \otimes \psi)(B_1))\ne 0$.
 The following theorem is an immediate application of \cite[Theorem~2.2]{Tu01}.

\begin{theorem}\label{thm:Tu22}
We have
\[ \tau(N,\gamma \otimes \psi) \doteq \prod\limits_{i=1}^3 \det((\gamma \otimes \psi)(B_i))^{(-1)^{i}}.
\]
In particular, $H_*(N;Q(R[F])^k)=0$ if and only if
$\det((\gamma \otimes \psi)(B_2))\ne 0$.

\end{theorem}

We now consider the case  that $N$ has non--empty toroidal boundary.
Let $X$ be a CW--complex with the following two properties:
\begin{enumerate}
\item $X$ is simple homotopy equivalent to a CW--complex of $N$,
\item $X$ has one 0--cell, $n$ 1--cells and  $n-1$ 2--cells.
\end{enumerate}
It is well--known that such a CW--structure exists.
For example, if $N$ is the complement of a non--split link $L\subset S^3$, then we can take $X$
to be the 2--complex corresponding to a Wirtinger presentation of $\pi_1(S^3\setminus \nu L)$.

We now consider
\[
0 \to  C_2(\tilde{X})
\xrightarrow{\partial_2} C_1(\tilde{X}) \xrightarrow{\partial_1}
C_0(\tilde{X}) \to 0
\]
where $C_0(\tilde{X})\cong \mathbb{Z}[\pi]$,
$C_i(\tilde{X})\cong \mathbb{Z}[\pi]^n$ and $C_i(\tilde{X})\cong \mathbb{Z}[\pi]^{n-1}$.
Let $A_i, i=1,2$ over $\mathbb{Z}[\pi]$ be the matrices corresponding to
the boundary maps $\partial_i:C_i\to C_{i-1}$. As above we can arrange that
 \[
A_1 = (1-h_1, 1-h_2, \ldots, 1-h_n)  \] where $\{h_1,\dots,h_n\}$
is a  generating set for $\pi$. Since $\psi$ is non--trivial
there exists an $s$ such that $\psi(h_s)\ne 0$.  Let $B_2$ be the result of deleting the $s$--th row from $A_2$. Let
$B_1$ be the $s$--th column of $A_1$.
As above we have $\det((\gamma \otimes \psi)(B_1))\ne 0$.
 The following theorem is again an immediate application of \cite[Theorem~2.2]{Tu01}.

\begin{theorem}\label{thm:Tu22boundary}
We have \[ \tau(N,\gamma \otimes \psi) \doteq \prod\limits_{i=1}^2 \det((\gamma \otimes \psi)(B_i))^{(-1)^{i}}.
\]
In particular, we have $H_*(N;Q(R[F])^k)=0$ if and only if
$\det((\gamma \otimes \psi)(B_2))\ne 0$.

\end{theorem}

\subsection{Torsion invariants}\label{section:torsioninvariants}
Let $S$ be a Noetherian unique factorization domain
(henceforth UFD). Examples of Noetherian UFD's are given by $\mathbb{Z}$ and by fields,
furthermore if $R$ is a Noetherian UFD and $F$ a free abelian group, then $R[F]$ is again a Noetherian UFD.

For a finitely generated
$S$-module $A$ we can find a presentation
$$
S^r \xrightarrow{P} S^s \to A \to 0
$$
since $S$ is Noetherian. Let $i\ge 0$ and suppose $s-i\le r$. We
define $E_i(A)$, \emph{the $i$-th elementary ideal} of $A$, to be
the ideal in $S$ generated by all $(s-i)\times (s-i)$ minors of
$P$ if $s-i>0$ and to be $S$ if $s-i\le 0$. If $s-i > r$, we
define $E_i(A)= 0$. It is known that $E_i(A)$ does not depend on
the choice of a presentation of $A$ (cf. \cite{CF77}).

Since $S$ is a UFD there exists a unique smallest principal ideal
of $S$ that contains $E_0(A)$. A generator of this principal ideal
is defined to be the \emph{order of $A$} and denoted by $\mbox{ord}
(A)\in S$. The order is well-defined up to multiplication by a
unit in $S$. Note that  $A$ is  $S$-torsion if and only if
$\mbox{ord} (A) \ne 0$. For more details, we refer to \cite{Tu01}.

\subsection{Twisted Alexander invariants}\label{section:twialex}
Let $N$ be a 3-manifold and let $\alpha:\pi_1(N)\to \mbox{GL}(k,R[F])$ be a representation with $R$ a Noetherian UFD.
Similarly to Section \ref{section:twitau} we
define the $R[F]$--chain complex $C_*(\tilde{N})\otimes_{\mathbb{Z}[\pi_1(N)]}R[F]^k$.
 For $i\ge 0$, we define \emph{the $i$-th twisted Alexander module of $(N,\alpha)$} to be the $R[F]$--module
$$
H_i(N;R[F]^k) \mathrel{\mathop:}= H_i(C_*(\tilde{N})\otimes_{\mathbb{Z}[\pi_1(N)]}R[F]^k).
$$
where $\pi_1(N)$ acts on $R[F]^k$ by $\alpha$.
 Since $N$ is compact and $R[F]$
is Noetherian these modules are finitely presented over $R[F]$.

\begin{definition} \label{def:polynomial}
The \emph{$i$-th twisted Alexander polynomial of $(N,\alpha)$}
is defined to be $\mbox{ord} (H_i(N;R[F]^k))\in R[F]$ and denoted by
$\Delta^{\alpha}_{N,i}$.
\end{definition}

Recall that by the discussion of Section \ref{section:torsioninvariants} twisted Alexander polynomials are well-defined up to
multiplication by a unit in $R[F]$.
Note that the units of $R[F]$ are of the form $rf$ with $r$ a unit in $R$ and $f\in F$.
 In the following, given $p\in R[F]$ we write
\[ \Delta^{\alpha}_{N,i} \doteq p \]
if there exists a representative of $\Delta^{\alpha}_{N,i}$ which equals $p$.

\noindent
We often write
 $\Delta^\alpha_{N}$ instead of  $\Delta^{\alpha}_{N,1}$,
and we refer to it as the \emph{twisted Alexander polynomial of
$(N,\alpha)$}. We recall that given a representation $\gamma:\pi_1(N)\to \mbox{GL}(k,R)$ and a non--trivial homomorphism $\psi:\pi_1(N)\to F$ to a free abelian group $F$ we get a tensor
representation $\gamma\otimes \psi$ and in particular  twisted Alexander polynomials
 $\Delta^{\gamma\otimes \psi}_{N,i}$. In almost all cases we will consider twisted Alexander polynomials corresponding to such a tensor representation.

 \noindent
When we consider twisted Alexander polynomials of links we  adopt the following notational conventions:
\begin{enumerate}
\item We identify $R[\mathbb{Z}]$ with $R[t^{\pm
1}]$ and $R[\mathbb{Z}^m]$ with $R[t_1^{\pm 1},\dots,t_m^{\pm 1}]$.
\item Given a link $L\subset S^3$ together with the abelianization  $\psi:\pi_1(S^3\setminus \nu L)\to F$
and a representation $\gamma:\pi_1(S^3\setminus \nu L)\to \mbox{GL}(k,R)$  with $R$ a Noetherian UFD, we  write
$\Delta^{\gamma}_{L,i}$ instead of $\Delta^{\gamma\otimes \psi}_{S^3\setminus \nu L,i}$.
\item If $L$ is an ordered oriented link, then we have a canonical isomorphism $F\cong \mathbb{Z}^m$
and we can identify $R[F]$ with $R[t_1^{\pm 1},\dots,t_m^{\pm 1}]$.
\item We sometimes record the fact that the twisted Alexander polynomial of a link $L$ is a (multivariable) Laurent polynomial in the notation, i.e. given
an oriented knot $K\subset S^3$ we sometimes write $\Delta^\gamma_{K,i}(t)=\Delta^\gamma_{K,i}\in R[t^{\pm
1}]$
and given an ordered oriented $m$--component link $L\subset S^3$
we sometimes write $\Delta^\gamma_{L,i}(t_1,\dots,t_m)=\Delta^\gamma_{L,i}\in R[t_1^{\pm 1},\dots,t_m^{\pm 1}]$.
\item Finally given a link $L$ we also drop the representation from the notation when the representation
 is the trivial representation to $\mbox{GL}(1,\mathbb{Z})$.
\end{enumerate}

With all these conventions, given a knot $K\subset S^3$, the polynomial $\Delta_K(t)=\Delta_K \in \mathbb{Z}[t^{\pm 1}]$
is just the ordinary Alexander polynomial.

\subsection{Computation of twisted Alexander polynomials}\label{section:compdelta}

Let $N$ be a 3--manifold with empty or toroidal boundary, and  ${\alpha}:\pi_1(N)\to \mbox{GL}(k,R[F])$  a  representation with $R$ a Noetherian UFD
and $F$ a free abelian group.
Given a finite presentation for $\pi_1(N)$ the polynomials $\Delta_{N,1}^{\alpha}\in R[F]$ and $\Delta_{N,0}^{\alpha}\in R[F]$ can be computed
efficiently using Fox calculus (cf. e.g. \cite[p.~98]{CF77} and \cite{KL99a}). We point out that
because we view $C_*(\tilde{N})$ as a \emph{right} module over $\mathbb{Z}[\pi_1(N)]$ we need a slightly
different definition of Fox derivatives than the one commonly used. We refer to \cite[Section~6]{Ha05} for details.
Finally, Proposition \ref{prop:dualitydelta}  allows us to compute $\Delta_{N,2}^{\alpha}\in R[F]$ using the algorithm for
computing the zeroth twisted Alexander polynomial.
In particular  all the  twisted Alexander polynomials $\Delta^{{\alpha}}_{N,i}$ can be computed from a finite presentation
of the fundamental group.

\section{Basic properties of twisted invariants}\label{section:basics}

In this section we summarize various basic algebraic properties of twisted Reidemeister torsion and twisted Alexander polynomials.

%
%
%
%

\subsection{Relationship between twisted invariants}

The following proposition is \cite[Theorem~4.7]{Tu01}.

\begin{proposition}\label{prop:deltataualpha}
Let $N$ be a 3--manifold with empty or toroidal boundary and let
$\alpha:\pi_1(N)\to \mbox{GL}(k,R[F])$  a  representation where $R$ is a Noetherian UFD and $F$ a free abelian group. If $\Delta_{N,i}^{\alpha} \ne 0$
for $i=0,1,2$, then
\[ \tau(N,{\alpha} )\doteq  \prod_{i=0}^2 \big(\Delta^{{\alpha}}_{N,i}\big)^{(-1)^{i+1}}.\]
\end{proposition}

The following is a mild extension of \cite[Proposition~2.5]{FK06}, \cite[Lemmas~6.2~and~6.3]{FK08a} and \cite[Theorem~6.7]{FK08a}.
Most of the ideas go back to work of Turaev (cf. e.g. \cite{Tu86} and \cite{Tu01}). The third statement is proved in \cite{DFJ10}.

\begin{proposition}\label{prop:deltatau}\label{prop:taudelta}
Let $N$ be a 3--manifold with empty or toroidal boundary,  $\psi:\pi_1(N)\to F$ a non--trivial homomorphism to a free abelian group $F$
and $\gamma:\pi_1(N)\to \mbox{GL}(k,R)$  a  representation where $R$ is a Noetherian UFD. Then the following hold:
\begin{enumerate}
\item $\Delta_{N,0}^{\gamma \otimes \psi} \ne 0$ and $\Delta_{N,i}^{\gamma \otimes \psi}=1$ for $i\geq 3$.
\item If $\mbox{rank}(\mbox{Im}\{\pi_1(N)\to F\})>1$, then $\Delta^{{\gamma \otimes \psi}}_{N,0} \doteq 1$.
\item If $\g$ is irreducible and if $\g$ restricted to $\ker(\psi)$ is non-trivial, then $\Delta^{{\gamma \otimes \psi}}_{N,0} \doteq 1$.
\item If $\Delta_{N,1}^{\gamma \otimes \psi} \ne 0$, then $\Delta_{N,2}^{\gamma \otimes \psi} \ne 0$.
\item If $N$ has non--empty boundary and if $\Delta_{N,1}^{\gamma \otimes \psi} \ne 0$, then $\Delta^{{\gamma \otimes \psi}}_{N,2} \doteq 1$.
\item If  $\mbox{rank}(\mbox{Im}\{\pi_1(N)\to F\})>1$ and if $\Delta_{N,1}^{\gamma \otimes \psi} \ne 0$, then $\Delta^{{\gamma \otimes \psi}}_{N,2}\doteq 1$.
\item We have $\Delta_{N,1}^{\gamma \otimes \psi}= 0$ if and only if
$ \tau(N,{\gamma \otimes \psi} )=0$.
\item If $\Delta_{N,1}^{\gamma \otimes \psi} \ne 0$, then
\[ \tau(N,{\gamma \otimes \psi} )\doteq  \prod_{i=0}^2 \big(\Delta^{{\gamma \otimes \psi}}_{N,i}\big)^{(-1)^{i+1}}.\]
\end{enumerate}
\end{proposition}

A few remarks regarding the equalities of Proposition \ref{prop:deltataualpha} and Proposition \ref{prop:deltatau} (8)  are in order:
\begin{enumerate}
\item Note that twisted Reidemeister torsion has in general a smaller indeterminacy than twisted Alexander polynomials.
In particular the equality
holds up to the indeterminacy of the twisted Alexander polynomials.
\item As pointed out in Section \ref{section:compdelta}, the twisted Alexander polynomials $\Delta^{\gamma\otimes \psi}_{N,i}$ can be computed from a presentation
of the fundamental group, whereas the computation of $\tau(N,\gamma \otimes \psi)$
requires in general an understanding of the CW--structure of $N$ (cf. Section \ref{section:comptau}). In particular
the equality of  Proposition \ref{prop:deltatau} (8) is often a faster method for computing $\tau(N,\gamma \otimes \psi)$ (at the price of a higher indeterminacy).
\item The twisted Alexander polynomial is only defined for representations over a Noetherian UFD, whereas the twisted Reidemeister torsion is defined for a finite dimensional representation over any commutative ring.
\item It is an immediate consequence of Proposition \ref{prop:deltatau} that $\tau(N,{\gamma \otimes \psi} )$ lies in $R[F]$, i.e. is a polynomial,
if $\mbox{rank}(\mbox{Im}\{\pi_1(N)\to F\})>1$.
\end{enumerate}

\begin{remark}
\begin{enumerate}
\item
Given a link $L\subset S^3$ and a representation $\gamma:\pi_1(S^3\setminus L)\to \mbox{GL}(k,R)$
Wada \cite{Wa94} introduced an invariant, which in this paragraph we refer to as $W(L,\gamma)$.
Wada's invariant is in many papers referred to as the twisted Alexander polynomial of a link.
Kitano \cite{Ki96} showed that $W(L,\gamma)$
agrees with the Reidemeister torsion $\tau(L,\gamma)$ (with the same indeterminacy). This can also be shown using the arguments of  Section \ref{section:comptau}.
In particular, in light of Proposition \ref{prop:deltatau} we see that $W({L,\gamma})\doteq \Delta_{L}^{\gamma}$ if $L$ has more than one component.
\item Lin's original definition \cite{Lin01} of the twisted Alexander polynomial of a knot uses `regular Seifert surfaces' and is rather different in character to the algebra--topological approach taken in the subsequent papers. The relation between the definitions of twisted Alexander polynomials given by Lin \cite{Lin01}, Jiang and Wang \cite{JW93} and Section \ref{section:twialex}
    is explained in \cite[Proposition~3.3]{JW93} and \cite[Section~4]{KL99a}.
\end{enumerate}
\end{remark}

\subsection{Twisted invariants for conjugate representations}
Given a group $\pi$ we say that two representations $\gamma_1,\gamma_2:\pi \to \mbox{GL}(k,R)$ are \emph{conjugate} if there exists
$P\in \mbox{GL}(k,R)$ such that $\gamma_1(g)=P\gamma_2(g)P^{-1}$ for all $g\in \pi$.
We recall the following elementary lemma.

\begin{lemma}\label{lem:equ:rep}
Let $N$ be a 3--manifold with empty or toroidal boundary and let $\psi:\pi_1(N)\to F$ a non--trivial homomorphism to a free abelian group.
If $\gamma_1$ and $\gamma_2$ are conjugate representations of $\pi_1(N)$,
then
\[ \tau(N,\gamma_1\otimes \psi)\doteq \tau(N,\gamma_2\otimes \psi),\]
if  $R$ is furthermore a Noetherian UFD, then for any $i$ we have
\[ \Delta_{N,i}^{\gamma_1\otimes \psi}\doteq \Delta_{N,i}^{\gamma_2\otimes \psi}. \]
\end{lemma}

Note that non--conjugate representations do not necessarily give different Alexander polynomials (cf. \cite[Theorem~B]{LX03}).
\subsection{Change of variables}
In this section we will show how to reduce the number of variables in twisted Alexander polynomials, in particular
this discussion will show how to obtain one--variable twisted Alexander polynomials from multi--variable twisted Alexander polynomials.

Throughout this section let $N$ be a 3--manifold with empty or toroidal boundary, let
$\psi:\pi_1(N)\to F$ be a non--trivial homomorphism to a free abelian group $F$
and let $\gamma:\pi_1(N)\to \mbox{GL}(k,R)$ be a  representation.
Furthermore let $\phi:F\to H$ also be a homomorphism to a  free abelian group such that $\phi\circ \psi$ is non--trivial.
We denote the induced ring homomorphism $R[F]\to R[H]$ by $\phi$ as well.
Let
\[ S=\{ f \in R[F] \, |\, \phi(R[F])\ne 0 \in R[H]\}.\]
Note that $\phi$ induces a homomorphism $R[F]S^{-1}\to Q(R[H])$ which we also denote by $\phi$.

The following is  a slight generalization of \cite[Theorem~6.6]{FK08a}, which in turn builds on ideas of Turaev
(cf. \cite{Tu86} and \cite{Tu01}).

\begin{proposition}\label{prop:changevartau}
We have  $\tau(N,\gamma \otimes \psi)\in R[F]S^{-1}$, and
\[ \tau(N,\gamma\otimes \psi \circ \phi) \doteq \phi(\tau(N,\gamma \otimes \psi)).\]
\end{proposition}

The following is now an immediate corollary of the previous proposition and
Proposition \ref{prop:deltatau}.

\begin{corollary}\label{cor:changevardelta}
Assume that  $\gamma:\pi_1(N)\to \mbox{GL}(k,R)$ is a  representation with $R$ a Noetherian UFD.
If $\phi\left(\Delta_{N}^{\gamma\otimes \psi}\right)\ne 0$, then
we have the following equality:
\[ \prod_{i=0}^2 \left(\Delta^{\gamma \otimes \phi\circ \psi}_{N,i}\right)^{(-1)^{i+1}}\doteq
\phi \left(\prod_{i=0}^2 \big(\Delta^{\gamma\otimes \psi}_{N,i}\big)^{(-1)^{i+1}}\right).\]
In particular, if $\mbox{rank}\{\mbox{Im}\{\pi_1(N)\xrightarrow{\phi\circ \psi}H\}\}\geq 2$, then
\[  \Delta^{\gamma \otimes \phi\circ \phi}_{N} \doteq \phi\left(\Delta^{\gamma\otimes \psi}_{N}\right). \]
\end{corollary}

\subsection{Duality for twisted invariants}\label{section:duality}

Let   $R$ be a ring with a (possibly trivial) involution $r\mapsto \overline{r}$.
Let $F$ be a free abelian group, with its natural involution. We extend the involution on $R$ to the group ring $R[F]$ and the quotient field $Q(R[F])$ in the usual way.
We equip $R[F]^k$ with the
standard hermitian inner product $\langle v ,w \rangle=v^t\overline{w}$ (where we view elements in $R[F]^k$
as column vectors).

Let $N$ be a 3--manifold with empty or toroidal boundary
and let $\alpha:\pi_1(N)\to \mbox{GL}(k,R[F])$ a representation.
We denote by  $\overline{\alpha} : \pi_1(N) \to \mbox{GL}(k,R[F])$  the representation given by
\[ \langle \alpha(g^{-1})v,w\rangle =\langle v,\overline{\alpha}(g)w\rangle \]
for all $v,w\in R[F]^k,g\in \pi_1(N)$. Put differently, for any $g\in \pi_1(N)$ we have
\[ \overline{\alpha}(g)= \overline{\left(\alpha(g)^{-1}\right)^t} \in \mbox{GL}(k,R[F]).\]
We say that a representation is \emph{unitary} if $\overline{\alpha}=\alpha$.

Note that if  $\psi:\pi_1(N)\to F$  is a non--trivial homomorphism to a free abelian group $F$
and  $\gamma:\pi_1(N)\to \mbox{GL}(k,R)$  a representation, then
\[ \overline{\gamma \otimes \psi} = \overline{\gamma}\otimes \psi.\]

The following duality theorem can be proved using the ideas of \cite{Ki96} and \cite[Corollary~5.2]{KL99a}.

\begin{proposition}   \label{prop:dualitytau}
Let $N$ be a 3--manifold with empty or toroidal boundary and let $\alpha:\pi_1(N)\to \mbox{GL}(k,R[F])$ a representation.
Then
\[ \overline{\tau(N,{\alpha})} \doteq  {\tau(N,\overline{a})}\in R[F].\]
In particular if  $\psi:\pi_1(N)\to F$  is a non--trivial homomorphism to a free abelian group $F$
and  $\gamma:\pi_1(N)\to \mbox{GL}(k,R)$  is a unitary  representation, then $\tau(N,\gamma\otimes \psi)$ is reciprocal, i.e.
\[ \overline{\tau(N,\gamma\otimes \psi)}=\tau(N,\gamma\otimes \psi)\in R[F].\]
\end{proposition}

It is easy to see that in general the twisted Reidemeister torsion is not reciprocal if one considers representations
$\alpha$
such that  $\alpha(g)\ne \overline{\alpha(g)}$ for some $g\in \pi_1(N)$.
Hillman, Silver and Williams \cite{HSW09} give much more subtle examples which show that there also exist knots $K$
together with special linear representations such that the corresponding twisted Reidemeister torsion  is not reciprocal.

The following proposition follows from the discussion in \cite{FK06}.

\begin{proposition} \label{prop:dualitydelta}
Let $N$ be a closed 3--manifold  and let $\alpha:\pi_1(N)\to \mbox{GL}(k,R[F])$ a representation with $R$  a  Noetherian UFD.
Assume that  $\Delta_{N,i}^{\alpha} \ne 0$ for all $i$. Then the following equalities hold:
\[ \overline{\Delta_{N,2}^{\alpha}}\doteq \Delta_{N,0}^{\overline{a}} \mbox{ and }  \overline{\Delta_{N}^{\alpha}}\doteq \Delta_{N}^{\overline{a}}.\]
\end{proposition}

\subsection{Shapiro's lemma for twisted invariants}

Let $N$ be a 3--manifold with empty or toroidal boundary.
Let $p:\hat{N}\to N$ be a finite cover of degree $d$.
Let $F$ be a free abelian group and let $\hat{\alpha}:\pi_1(\hat{N})\to \mbox{GL}(k,R[F])$ a representation.

 Now consider  the $R[F]^k$--module
 \[ \mathbb{Z}[\pi_1({N})]\otimes_{\mathbb{Z}[\pi_1(\hat{N})]} R[F]^k.\]
 If $g_1,\dots,g_d$ are representatives of $\pi_1(N)/\pi_1(\hat{N})$ and if $e_1,\dots,e_k$ is the canonical basis of $R[F]^k$,
 then it is straightforward to see that the above $R[F]$--module is a free $R[F]$--module with basis $g_i\otimes e_j, i=1,\dots,d, j=1,\dots,k$.
 The group $\pi_1(N)$ acts on
\[ \mathbb{Z}[\pi_1({N})]\otimes_{\mathbb{Z}[\pi_1(\hat{N})]} R[F]^k =R[F]^{kd}\]
via left multiplication which defines a representation $\pi_1(N)\to \mbox{GL}(kd,R[F])$ which we denote by $\alpha$.

\begin{remark}
\begin{enumerate}
\item
Let  $\hat{\gamma}:\pi_1(\hat{N})\to \mbox{GL}(k,R)$  be a  representation.
Let  $\psi:\pi_1(N)\to F$ be a non--trivial homomorphism to a free abelian group $F$.
We write $\hat{\psi}=\psi\circ p_*$ and we  denote by $\gamma$ the
 representation given by left multiplication by $\pi_1(N)$ on
\[ \mathbb{Z}[\pi_1({N})]\otimes_{\mathbb{Z}[\pi_1(\hat{N})]} R^k =R^{kd}.\]
Given $\hat{\alpha}=\hat{\gamma}\otimes \hat{\psi}$ we have in that case $\alpha=\gamma \otimes \psi$.
\item If $\hat{\gamma}$ is the trivial one--dimensional representation, and $\hat{N}$ the cover of $N$ corresponding to an epimorphism
$\varphi:\pi_1(N)\to G$ to a finite group, then we have $\gamma=\varphi$.
\end{enumerate}
\end{remark}

The following is now a variation on Shapiro's lemma (cf. \cite[Lemma~3.3]{FV08a} and \cite[Section~3]{HKL08}).

\begin{theorem}\label{thm:shapiro}
We have
\[ \tau(\hat{N},\hat{\alpha}) \doteq \tau(N,\alpha).\]
If $R$ is furthermore a  Noetherian UFD, then
\[ \Delta_{\hat{N},i}^{\hat{\alpha}} \doteq  \Delta_{N,i}^{\alpha}.\]
\end{theorem}

In its simplest form Theorem \ref{thm:shapiro} says that given an
epimorphism
$\gamma:\pi_1(N)\to G$ to a finite group the corresponding twisted Alexander polynomials of $N$ are just  untwisted Alexander polynomials of the corresponding finite cover.

\subsection{Twisted invariants of knots and links}
Let $L=L_1\cup\dots \cup L_m\subset S^3$ be an ordered oriented link.
Recall that given a representation $\gamma:\pi_1(S^3\setminus \nu L)\to \mbox{GL}(k,R)$ with $R$ a Noetherian UFD we can consider the
twisted Alexander polynomial $\Delta_L^\gamma$ as an element in the Laurent polynomial ring
$R[t_1^{\pm 1},\dots,t_m^{\pm 1}]$ and we
 write $\Delta^\gamma_{L,i}(t_1,\dots,t_m)=\Delta^\gamma_{L,i}\in R[t_1^{\pm 1},\dots,t_m^{\pm 1}]$.

 Given  $\epsilon_j \in \{ \pm 1\}$ for $j=1,\dots,m$ we denote by
$L^\epsilon$ the link $\epsilon_1 L_1,\dots,\epsilon_m  L_m$, i.e. the oriented link obtained from $L$ by reversing  the orientation
of all components with $\epsilon_j=-1$.

The following lemma is now an immediate consequence of the definitions:

\begin{lemma}\label{lem:reverse}
Given  $\epsilon_j \in \{ \pm 1\}$ for $j=1,\dots,m$ we have
\[ \Delta^\gamma_{L^\epsilon,i}(t_1,\dots,t_m)=\Delta^\gamma_{L,i}(t_1^{\epsilon_1},\dots,t_m^{\epsilon_m}).\]
\end{lemma}

We now turn to the study of twisted Alexander polynomials of sublinks.
Given a link $L=L_1\cup \dots L_{k-1}\cup L_k\subset S^3$ Torres \cite{To53} showed how to relate the
Alexander polynomials of $L$ and $L'=L_1\cup \dots L_{k-1}\subset S^3$. The following theorem of Morifuji \cite[Theorem~3.6]{Mo07} gives a generalization of the Torres condition to twisted Reidemeister torsion.

\begin{theorem}
Let $L=L_1\cup \dots L_{k-1}\cup L_k\subset S^3$ be a link. Write  $L'=L_1\cup \dots L_{k-1}\subset S^3$
and let $\gamma':\pi_1(S^3\setminus \nu L')\to \mbox{SL}(n,\mathbb{F})$ be a representation where $\mathbb{F}$ is a field.
Denote by $\gamma$ the representation $\pi_1(S^3\setminus \nu L)\to \pi_1(S^3\setminus \nu L')\to \mbox{SL}(n,\mathbb{F})$,
then
\[ \tau(L,\gamma)(t_1,\dots,t_{k-1},1) \doteq \left(T^n+\sum_{k=1}^{n-1}\epsilon_iT^i+(-1)^n\right)\cdot \tau(L',\gamma')(t_1,\dots,t_{k-1}) \]
where
\[ T\mathrel{\mathop:}=\prod_{i=1}^{k-1} t_i^{lk(L_i,L_k)} \]
and where $\epsilon_1,\dots,\epsilon_{n-1}$ are elements of $\mathbb{F}$.
\end{theorem}

\section{Distinguishing knots and links}\label{section:dist}

In this section we will restrict ourselves to twisted Alexander polynomials of knots and links.
Recall that given an oriented knot $K\subset S^3$ the \emph{reverse} $\overline{K}$ is given by reversing the orientation.
Given a knot $K\subset S^3$ we denote by $K^*$ its \emph{mirror image}, i.e. the result of reflecting $K$ in  $S^2\subset S^3$.
The mirror image is also sometimes referred to as the \emph{obverse}.

Twisted Alexander invariants have so far been surprisingly little used to distinguish a knot from its mirror image or from its reverse
(cf. \cite{KL99b} though for a deep result showing that the knot $8_{17}$ and its reverse lie in different concordance classes).
It is an interesting question whether  Kitayama's normalized Alexander polynomial \cite{Kiy08a} can be used to distinguish a knot from its mirror image and its reverse.
We  refer to \cite{Ei07} for an interesting and very successful approach to distinguishing knots using `knot colouring polynomials'.

In this section we are from now on only concerned with distinguishing knot types of prime knots.
Here we say that two knots $K_1$ and $K_2$ are of the same knot type if there exists a homeomorphism $h$ of the sphere
with $h(K_1)=K_2$. Put differently, $K_1$ and $K_2$ are of the same knot type if they are related by  an isotopy
together possibly with taking the mirror image and possibly reversing the orientation.

The most common approach for distinguishing knots using twisted Alexander polynomials is to look at the set (or product) of all twisted Alexander polynomials
corresponding to a `characteristic set' of representations.
Note that due to Lemma \ref{lem:equ:rep} we can in fact restrict ourselves to conjugacy classes of a set of characteristic representations.

The following two types of characteristic sets have been used in the literature:
\begin{enumerate}
\item Given a knot $K$ consider
all conjugacy classes  of (all upper triangular, parabolic, metabelian, orthogonal, unitary) representations of $\pi_1(S^3\setminus \nu K)$ of a fixed dimension over a finite ring.
\item Given $K$ consider all conjugacy classes of homomorphisms of $\pi_1(S^3\setminus \nu K)$ onto a finite group $G$ composed with a fixed representation of $G$.
\end{enumerate}

%

The first approach was used in Lin's original paper \cite{Lin01} to distinguish knots with the same Alexander module.
Wada \cite{Wa94} also used the first approach to show that twisted Alexander polynomials can distinguish the Conway knot and the Kinoshita--Terasaka knot (cf. also \cite{In00}).
This shows in particular that twisted Alexander polynomials detect mutation.
In fact in \cite{FV07a} it is shown that twisted Alexander polynomials detect all mutants with 11 crossings or less.
Furthermore, in \cite{FV07a} the authors give an example of a pair of knot types of prime knots which can be distinguished using twisted Alexander polynomials,
even though the HOMFLY polynomial, Khovanov homology and knot Floer homology agree.

\begin{remark}
\begin{enumerate}
\item
The approach of using twisted Alexander polynomials corresponding to a characteristic set of conjugacy classes can be viewed as an extension
of the approach of Riley \cite{Ri71} who studied the first homology group corresponding to such a set of representations to distinguish the Conway knot from the Kinoshita--Terasaka knot.
\item By the work of Whitten \cite{Wh87} and Gordon--Luecke \cite{GL89} the knot type of a prime knot is determined by its fundamental group.
It is therefore at least conceivable that twisted Alexander polynomials can distinguish any two pairs of knot types.
\end{enumerate}
\end{remark}

The following theorem shows that twisted Alexander polynomials detect the unknot and the Hopf link.
The statement for knots was first proved by Silver and Williams \cite{SW06}, the extension to links was later proved
 in \cite{FV07a}.

\begin{theorem}\label{thm:swfv}
Let $L\subset S^3$ be a  link which is neither the unknot nor  the Hopf link. Then there exists an epimorphism $\gamma:\pi_1(S^3\setminus \nu L)\to G$ onto a finite group
$G$ such that $\Delta^\gamma_L \not\doteq  \pm 1$.
\end{theorem}

Given a knot $K$ we denote by $t(K)$ its tunnel number.
Theorem \ref{thm:swfv} was used by Pajitnov \cite{Pa08} to show that for any knot $K$ there exists $\lambda>0$
such that $t(nK)\geq \lambda n-1$.

\section{Twisted Alexander polynomials and concordance}\label{section:conc}

We first recall the relevant definitions.
Let $L=L_1\cup \dots \cup L_m\subset S^3$ be an oriented $m$--component link.
We say that $L$ is (topologically) \emph{slice} if the components bound $m$ disjointly embedded locally flat disks in $D^4$.
Given two ordered oriented $m$--component links $K=K_1\cup \dots \cup K_m\subset S^3$ and $L=L_1\cup \dots \cup L_m\subset S^3$
we say that $K$ and $L$ are \emph{concordant} if there exist $m$ disjointly embedded locally flat cylinders $C_1,\dots,C_m$ in $S^3\times [0,1]$
such that $\partial C_i=K_i\times 0\, \cup \, -L_i\times 1$. (Given an oriented knot $K\subset S^3$ we write $-K=\overline{K}^*$, i.e. $-K$ is  the knot obtained from the mirror image of $K$ by reversing the orientation.) Note that two knots $K_1$ and $K_2$ are concordant
if and only if $K_1\# -K_2$ is slice, and a link is slice if and only if it is concordant to the unlink.

\subsection{Twisted Alexander polynomials of zero--surgeries}\label{section:zerosurgery}

Given a knot $K\subset S^3$ the zero--framed surgery $N_K$ of $S^3$ along $K$ is defined to be
\[ N_K =S^3\setminus \nu K \cup_{T} S^1\times D^2\]
where $T=\partial (S^3\setminus \nu K)$ and $T$ is glued to $S^1\times D^2$ by gluing the meridian to $S^1\times \mbox{pt}$.
The inclusion map induces an isomorphism $\mathbb{Z}\cong H_1(S^3\setminus \nu K;\mathbb{Z})\xrightarrow{\cong} H_1(N_K;\mathbb{Z})$.
It is well--known that understanding the zero--framed surgery $N_K$
is the key to determining whether $K$ is slice or not
(cf. e.g. \cite{COT03}, \cite[Proposition~3.1]{FT05}, \cite[Proposition~2.1]{CFT09}).
We therefore prefer to formulate the sliceness obstructions in terms of the twisted Alexander polynomials of the zero--framed surgery.
The following lemma, together with Proposition \ref{prop:taudelta}, relates the twisted invariants of the zero--framed surgery with the twisted invariants of the knot complement.
We refer to \cite[Lemma~6.3]{KL99a} and \cite[Section~VII]{Tu02a} for very similar statements.

\begin{lemma}\label{lem:zerosurgery}
Let $K \subset S^3$ be a knot. Denote its meridian by $\mu$.
Let  $\alpha:\pi_1(N_K)\to \mbox{GL}(k,R[t^{\pm 1}])$ be a representation
such that $\det(\alpha(\mu)-\mbox{id})\ne 0$.
We denote the inclusion induced representation $\alpha:\pi_1(S^3\setminus \nu K)\to \pi_1(N_K)\to \mbox{GL}(k,R[t^{\pm 1}])$ by $\alpha$ as well.
Then
\[ \tau(S^3\setminus \nu K,\alpha)=\tau(N_K,\alpha) \cdot \det(\alpha(\mu)-\mbox{id}).\]
\end{lemma}

\begin{proof}
The proof of the lemma is standard and well--known. We therefore give just a quick summary.
Consider the decomposition $ N_K = S^3\setminus \nu K \cup_{T} S$ as above, where $S = S^1 \times D^2$.
Using the Mayer--Vietoris theorem for torsion we obtain that
\[ \tau(N_K,\alpha) =\frac{\tau(S^3\setminus \nu K,\alpha)\cdot \tau(S,\alpha)} { \tau( T,\alpha)}.\]
It is well--known that the torsion of the torus is trivial (c.f. e.g \cite{KL99a})
and that the torsion of $S$ is given by
\[ \tau(S,\alpha) = \frac{1}{ \det(\alpha(\mu)-\mbox{id})}.\]
The lemma now follows immediately.
\end{proof}

%

\subsection{Twisted Alexander polynomials and knot concordance}\label{section:slice}

The first significant result in the study of slice knots is due to
Fox and Milnor \cite{FM66} who showed that if $K$ is a slice knot, then $\Delta_K \doteq f(t)f(t^{-1})$ for some $f(t)\in \mathbb{Z}[t^{\pm 1}]$.

In this section we will give an exposition and a slight generalization of the Kirk--Livingston \cite{KL99a} sliceness obstruction theorem, which generalizes the Fox--Milnor condition. Note that we will state this obstruction using a slightly different setup, but Theorem \ref{thm:slice} is
already contained in \cite{KL99a} and \cite{HKL08}.

Let $K$ be a knot, as above we  denote by $N_K$ the zero--framed surgery of $S^3$ along $K$.
Now suppose that $K$ has a  slice disk $D$. Note that $\partial (D^4\setminus \nu D)=N_K$.
Let $R$ be a ring with (possibly trivial) involution and as usual we extend the involution to $R[t^{\pm 1}]$ by $\overline{t}\mathrel{\mathop:}=t^{-1}$.
Let  $\alpha:\pi_1(N_K)\to \mbox{GL}(R[t^{\pm 1}],k)$ be a unitary  representation.
Assume that $\alpha$ has  the following two properties:
\begin{enumerate}
\item $\alpha$ factors through a representation $\pi_1(D^4\setminus \nu D)\to \mbox{GL}(k,R[t^{\pm 1}])$, and
\item the induced twisted modules $H_*(D^4\setminus \nu D;R[t^{\pm 1}]^k)$ are $R[t^{\pm 1}]$--torsion,
\end{enumerate}
then  using Poincar\'e duality for Reidemeister torsion (cf. \cite[Theorem~6.1]{KL99a}) one can show that  $\tau(N_K,\alpha)\doteq f(t)\overline{f(t)}$ for some $f(t)\in R(t)$. We refer to \cite[Corollary~5.3]{KL99a} for details.

\begin{remark}
Note that the canonical representation $\pi_1(N_K)\to \mbox{GL}(1,\mathbb{Z}[t^{\pm 1}])$ extends over any slice disk complement.
It can be shown that  $H_*(D^4\setminus \nu D;\mathbb{Z}[t^{\pm 1}])$ is $\mathbb{Z}[t^{\pm 1}]$--torsion. We can therefore recover the Fox--Milnor theorem from this discussion.
\end{remark}

Most of the ideas and techniques of finding  representations satisfying (1) and (2) go back to the seminal work of Casson and Gordon \cite{CG86}.
We follow the approach taken in \cite{Fr04} which is inspired by Letsche \cite{Let00} and Kirk--Livingston \cite{KL99a,HKL08}.

In the following let $K$ again be a knot in $S^3$.
We denote by $W_k$ the  cyclic $k$--fold branched cover of $K$.
Note that $H_1(W_k;\mathbb{Z})$ has a natural $\mathbb{Z}/k$--action, and we can therefore view $H_1(W_k;\mathbb{Z})$ as a $\mathbb{Z}[\mathbb{Z}/k]$--module.
If $H_1(W_k;\mathbb{Z})$ is finite, then there exists  a non--singular linking form
\[\lambda_k:H_1(W_k;\mathbb{Z})\times H_1(W_k;\mathbb{Z})\to \mathbb{Q}/\mathbb{Z}\]
with respect to which $\mathbb{Z}/k$ acts via isometries. We say that $M\subset H_1(W_k;\mathbb{Z})$ is a \emph{metabolizer} of the linking form
if $M$ is a $\mathbb{Z}[\mathbb{Z}/k]$--submodule of $H_1(W_k;\mathbb{Z})$ such that $\lambda_k(M,M)=0$ and such that $|M|^2=|H_1(W_k;\mathbb{Z})|$.
It is well--known that if $K$ is slice, then $\lambda_k$ has a metabolizer for any prime power $k$. We refer to \cite{Go78} for details.

Note that $H^1(N_K;\mathbb{Z})=\mathbb{Z}$, in particular $H^1(N_K;\mathbb{Z})$ has a unique generator (up to sign) which we denote by $\phi$.
We now consider the Alexander module  $H_1(N_K;\mathbb{Z}[t^{\pm 1}])$ which we denote by $H$.
Note that $H$ is isomorphic to the usual Alexander module of $K$.
It is well--known that given $k$ there exists a canonical isomorphism
$H/(t^k-1)\to H_1(W_k;\mathbb{Z})$ (cf. e.g. \cite[Corollary~2.4]{Fr04} for details).
Now let $\mu \in \pi_1(N_K)$ be an element  with $\phi(\mu)=1$.
Note that for any $g\in \pi$ we have $\phi(\mu^{-\phi(g)}g)=0$, in particular we can consider its image $[\mu^{-\phi(g)}g]$ in the abelianization $H_1(\ker(\phi))$, which we can identify with $H$.
Then we have a well--defined map
\[ \pi \to \mathbb{Z} \ltimes H \to \mathbb{Z} \ltimes H/(t^k-1) \]
where the first map is given by sending $g\in \pi$ to $(\phi(g),[\mu^{-\phi(g)}g])$. Here $n\in \mathbb{Z}$ acts on $H$ and on $H/(t^k-1)$ via multiplication by $t^n$.
We refer to \cite{Fr04} and \cite{BF08} for details.

Fix $k \in \mathbb{N}$.
Let  $\chi:H_1(W_k;\mathbb{Z})\to \mathbb{Z}/q\to S^1$ be a character. We denote the induced character $H\to H/(t^k-1)=H_1(W_k;\mathbb{Z})\to S^1$ by $\chi$ as well.
Let $\zeta_q$ be a primitive $q$--th root of unity.
Then it is straightforward  to verify that
 $$ \begin{array}{rcl} \alpha(k,\chi) :\pi \to  \mathbb{Z}\ltimes H/(t^{k}-1)& \to &\mbox{GL}(k,\mathbb{Z}[\zeta_q][t^{\pm 1}]) \\
    (j,h) &\mapsto &
\begin{pmatrix}
 0& \dots &0&t \\
 1&0& \dots &0 \\
\vdots &\ddots &  & \vdots\\
    0&\dots &1&0 \end{pmatrix}^j
\begin{pmatrix}
\chi(h) &0& \dots&0 \\
 0&\chi(th) & \dots& 0\\
\vdots&&\ddots & \vdots\\ 0&0&\dots &\chi(t^{k-1}h) \end{pmatrix} \end{array} $$
defines a unitary representation.
(Note the ``$t$'' in the upper right corner.)
Also note that $\alpha(k,\chi)$ is not a tensor representation (cf. \cite{HKL08}).

We can now formulate the following obstruction theorem which is well--known to the experts. It can be proved using the above discussion,
Proposition \ref{prop:deltataualpha}, various
well--known arguments going back to Casson and Gordon \cite{CG86} and \cite[Lemma~6.4]{KL99a}.
We also refer to Letsche \cite{Let00} and \cite{Fr04}  for more information.

\begin{theorem} \label{thm:slice}
Let $K$ be a slice knot. Then for any prime power $k$ there exists a metabolizer $M$ of $\lambda_k$
such that for any odd prime power $n$ and any  character $\chi:H_1(W_k;\mathbb{Z})\to \mathbb{Z}/n\to S^1$ vanishing on $M$ we have
\[ \Delta_{N_K}^{\alpha(k,\chi)}\doteq f(t)\overline{f(t)} \]
for some $f(t)\in Q(\mathbb{Z}[\zeta_n])[t^{\pm 1}]=\mathbb{Q}(\zeta_n)[t^{\pm 1}]$.
\end{theorem}

Note that  the original sliceness obstruction of Kirk and Livingston \cite[Theorem~6.2]{KL99a} (cf. also \cite[Theorem~8.1]{HKL08} and \cite[Theorem~5.4]{Liv09}) gives an obstruction in terms of twisted Alexander polynomials of
a cyclic cover of $S^3\setminus \nu K$ corresponding to one--dimensional representations. Using
Theorem \ref{thm:shapiro} and Lemma \ref{lem:zerosurgery} (cf. also \cite[Lemma~6.3]{KL99a}) one can show that the sliceness obstruction of Theorem \ref{thm:slice}
is equivalent to the original formulation by Kirk and Livingston.

\begin{remark}
In Theorem \ref{thm:slice} and later in Theorems \ref{thm:slicetensor} and \ref{thm:doublyslice} we restrict ourselves to characters of odd prime power.
Similar theorems also hold for characters of even prime power order, we refer to \cite[Lemma~6.4]{KL99a} and \cite[Section~5]{Liv09} for more information.
\end{remark}

Theorem \ref{thm:slice} can be somewhat generalized using tensor representations.
In the following let $k_1,\dots,k_n \in \mathbb{N}$. Assume we are given  characters $\chi_i:H_1(W_{k_i};\mathbb{Z})\to \mathbb{Z}/q\to S^1, i=1,\dots,n$
we also get a tensor representation
\[ \alpha(k_1,\chi_1)\otimes \dots \otimes \alpha(k_n,\chi_n)\to \mbox{GL}(k_1\cdot\dots \cdot k_n,\mathbb{Z}[\zeta_q][t^{\pm 1}]).\]
We refer to \cite[Proposition~4.6]{Fr04} for more information.

The following theorem can be proved by modifying the proof of Theorem \ref{thm:slice} along the lines of
\cite[Theorem~4.7]{Fr04}.

\begin{theorem} \label{thm:slicetensor}
Let $K$ be a slice knot. Let $q$ be an odd prime power and let $k_1,\dots,k_n$ be coprime prime powers.
Then there exist metabolizers $M_1,\dots,M_n$ of the linking forms $\lambda_{k_1},\dots,\lambda_{k_n}$
such that for any choice of  characters $\chi_i:H_1(W_{k_i};\mathbb{Z})\to \mathbb{Z}/q\to S^1,i=1,\dots,n$  vanishing on $M_i$ we have
\[ \Delta_{N_K}^{\alpha(k_1,\chi_1)\otimes \dots \otimes \alpha(k_n,\chi_n)}\doteq f(t)\overline{f(t)} \]
for some $f(t)\in Q(\mathbb{Z}[\zeta_q])[t^{\pm 1}]=\mathbb{Q}(\zeta_q)[t^{\pm 1}]$.
\end{theorem}

\begin{remark}
 It was first observed by Letsche \cite{Let00} that non--prime power dimensional representations can give rise to sliceness obstructions. It is shown in \cite{Fr03}
that Letsche's non--prime power representations are given by the tensor representations considered in Theorem \ref{thm:slicetensor}.
\end{remark}

The sliceness obstruction coming from twisted Alexander polynomials is in some sense less powerful than
the invariants of Casson and Gordon \cite{CG86} (cf. \cite[Section~6]{KL99a} for a careful discussion) and Cochran--Orr--Teichner \cite{COT03}.
But to date  twisted Alexander polynomials give the strongest sliceness obstruction for algebraically slice knots
that can be computed efficiently.
The Kirk--Livingston sliceness obstruction theorem  has been used by various authors to produce many interesting examples,
some of which we list below:
\begin{enumerate}
\item Kirk and Livingston \cite{KL99b} apply twisted Alexander polynomials to show that
some knots (e.g. $8_{17}$) are not concordant to their inverses. This shows in particular that
knots are not necessarily concordant to their mutants. In \cite{KL99b} it is also shown that in general a knot
is not even concordant to a positive mutant.
\item Tamulis \cite{Tam02} considered knots with at most ten crossings which have algebraic concordance order two.
Tamulis used twisted Alexander polynomials to show that all but one of these knots do not have order two in the knot concordance group.
\item
In \cite{HKL08} Herald, Kirk and Livingston consider all knots with up to twelve crossings.
Eighteen of these knots are algebraically slice but can not be shown to be slice using elementary methods.
In \cite{HKL08} twisted Alexander polynomials are used to show that sixteen of these knots are in fact not slice
and one knot is smoothly slice. Therefore among the knots with up to twelve crossings only the sliceness status of the knot $12_{a631}$ is unknown.
\item The concordance genus $g_c(K)$ of a knot $K$ is defined to be the minimal genus among all knots
concordant to $K$. Livingston \cite{Liv09} uses twisted Alexander polynomials to show that
the concordance genus of $10_{82}$ equals three, which is its ordinary genus.
\end{enumerate}

\subsection{Ribbon knots  and doubly slice knots}\label{section:ds}

A knot $K$ is called \emph{homotopy ribbon} if there exists a slice disk $D$ such that
$\pi_1(N_K)\to \pi_1(D^4\setminus \nu D)$ is surjective. Note that if a knot is ribbon, then it is also homotopy ribbon.
It is an open conjecture whether every knot that is slice is also homotopy ribbon.

Let $K\subset S^3$ be a knot. Recall that there exists a non--singular hermitian pairing
\[ \lambda:H_1(N_K;\mathbb{Z}[t^{\pm 1}])\times H_1(N_K;\mathbb{Z}[t^{\pm 1}])\to \mathbb{Q}(t)/\mathbb{Z}[t^{\pm 1}] \] which is referred to as the \emph{Blanchfield pairing}.
We say that $M\subset H_1(N_K;\mathbb{Z}[t^{\pm 1}])$ is a \mbox{metabolizer} for the Blanchfield pairing if $M=M^{\perp}$, i.e. if $M$ satisfies
\[ M=\{ x\in H_1(N_K;\mathbb{Z}[t^{\pm 1}]) \, |\, \lambda(x,y)=0 \in \mathbb{Q}(t)/\mathbb{Z}[t^{\pm 1}] \mbox{ for all }y\in M\}.\]
If $K$ is slice, then there exists a metabolizer for the Blanchfield pairing.
Also recall from the previous section that for any $k$ there exists a canonical isomorphism $H_1(N_K;\mathbb{Z}[t^{\pm 1}])/(t^k-1)=H_1(L_k;\mathbb{Z})$.
If $M$ is a metabolizer for the Blanchfield pairing, then for any $k$ we have that $M/(t^k-1)\subset H_1(N_K;\mathbb{Z}[t^{\pm 1}])/(t^k-1)=H_1(L_k;\mathbb{Z})$
is a metabolizer for the linking form $\lambda_k$.
We refer to \cite{Bl57}, \cite{Ke75}, \cite[Section~7]{Go78}, \cite{Let00} and \cite[Section~2.3]{Fr04} for more information.

The proof of Theorem \ref{thm:slice}  can now be modified along well established lines (cf. e.g. \cite[Theorem~8.3]{Fr04}) to prove the following theorem:

\begin{theorem}\label{thm:ribbon}
Let $K\subset S^3$ be a ribbon knot, then there exists a metabolizer $M\subset H_1(N_K;\mathbb{Z}[t^{\pm 1}])$ for the Blanchfield pairing $\lambda$
such that for any $k$ and any non--trivial character $\chi:H_1(L_{k};\mathbb{Z})\to S^1$ of prime power vanishing on $M/(t^k-1)$ we have
\[ \Delta_{N_K}^{\alpha(k,\chi)}\doteq f(t)\overline{f(t)} \]
for some $f(t)\in Q(\mathbb{Z}[\zeta_q])[t^{\pm 1}]=\mathbb{Q}(\zeta_q)[t^{\pm 1}]$.
\end{theorem}

Using \cite[Proposition~4.6]{Fr04}  one can show that if a knot $K$ satisfies the ribbon obstruction of Theorem \ref{thm:ribbon}, then it also satisfies the
sliceness obstruction given by Theorem \ref{thm:slicetensor}.

A knot $K\subset S^3$ is called \emph{doubly slice} if there exists an unknotted locally flat
two--sphere $S
\subset S^4$ such that
$S\cap S^3=K$.
Note that a doubly slice knot is in particular slice.
The ordinary Alexander polynomial does not contain enough information to distinguish between slice and doubly slice knots.
On the other hand twisted Alexander polynomials can detect the difference.
The following theorem is well--known to the experts. It can be proved using the above ideas
of Kirk and Livingston combined with the results of Gilmer and Livingston \cite{GL83} (cf. also \cite[Section~8.2]{Fr04}).
Note that many of the ideas already go back to the original paper by Sumners \cite{Sum71}.

\begin{theorem} \label{thm:doublyslice}
Let $K$ be a doubly slice knot. Then for any prime power $k$ there exist
two metabolizers $M_1$ and $M_2$ of $\lambda_k$ with $M_1\cap M_2=\{0\}$
such that for any odd prime power $n$ and any  character $\chi:H_1(L_k;\mathbb{Z})\to \mathbb{Z}/n\to S^1$ which vanishes either on $M_1$ or on $M_2$
we have
\[ \Delta_{N_K}^{\alpha(k,\chi)}\doteq f(t)\overline{f(t)} \]
for some $f(t)\in Q(\mathbb{Z}[\zeta_n])[t^{\pm 1}]$.
\end{theorem}

It is also possible to state and prove an analogue of Theorem \ref{thm:slicetensor} for doubly slice knots.


\subsection{Twisted invariants and slice links}\label{section:slicelinks} \label{section:linkslice}

Kawauchi \cite[Theorems~A~and~B]{Ka78} showed that if $L$ is a slice link with more than one component, then the ordinary Alexander module  is necessarily non--torsion,
in particular the corresponding Reidemeister torsion is zero.
We will follow an idea of Turaev \cite[Section~5.1]{Tu86} to define a (twisted) invariant for links even if the (twisted) Alexander module is non--torsion.
This invariant will then give rise to a sliceness obstruction for links. We refer throughout this section to \cite{CF10} and \cite{Tu86} for details.

Let $L\subset S^3$ be an oriented $m$--component link.
Let $R\subset \mathbb{C}$ be a subring and let $\alpha:\pi(S^3\setminus \nu L)\to \mbox{GL}(k,R)$ be a unitary representation. Suppose that
$\psi:\pi_1(S^3\setminus \nu L)\to F$ is a homomorphism to a free abelian group which is non--trivial on each meridian of $L$.
Under these assumptions      we can endow $H_1(S^3\setminus \nu L;Q(R[F])^k)$ and $H_2(S^3\setminus \nu L;Q(R[F])^k)$
with dual bases and using these bases we can define
 the
Reidemeister torsion
\[ \tilde{\tau}^{\alpha\otimes \psi}(L)\in Q(F)^*/N(Q(F)^*),\]
here  $N(Q(F)^*)$ denotes the subgroup of norms of the multiplicative group $Q(F)^*$, i.e.
 $N(Q(F)^*)=\{ q\overline{q} \, |\, q\in Q(F)^*\}$. Reidemeister torsion
$\tilde{\tau}^{\alpha\otimes \psi}(L)$ viewed as an element in $Q(F)^*/N(Q(F)^*)$ is well--defined up to multiplication
by an element of the form $\pm df$ where $d\in \det(\alpha(\pi_1(S^3\setminus \nu L))), f\in F$.
The invariant $\tau^{\alpha\otimes \psi}(L)$ is the twisted version of an invariant first introduced by Turaev \cite[Section~5.1]{Tu86}.

For example, if $L$ is the $m$--component unlink in $S^3$ with meridians $\mu_1,\dots,\mu_m$,
then given $\alpha$ and $\psi$ as above we have
\[ \tilde{\tau}^{\alpha\otimes \psi}(L)=\pm df \cdot \prod\limits_{i=1}^m \det\big(\mbox{id}-\psi(\mu_i)\alpha(\mu_i)\big)^{-1} \, \in Q(F)^*/N(Q(F)^*)\]
with $d\in \det(\alpha(\pi_1(S^3\setminus \nu L))), f\in F$.


In \cite{CF10} the first author and Jae Choon Cha prove the following result.

\begin{proposition}\label{prop:slicelink}
Let $L$ be an $m$--component oriented slice link with meridians $\mu_1,\dots,\mu_m$.
 Let $R\subset \mathbb{C}$ be a subring closed under complex conjugation and let $\alpha:\pi_1(S^3\setminus \nu L)\to \mbox{GL}(k,R)$ be a representation which factors through a finite group of prime power order. Let
$\psi:H_1(S^3\setminus \nu L)\to F$ be an epimorphism onto a free abelian group which is non--trivial on each meridian of $L$.  Then
\[ \tilde{\tau}^{\alpha\otimes \psi}(L)=\pm
 df\cdot \prod\limits_{i=1}^m \det\big(\mbox{id}-\psi(\mu_i)\alpha(\mu_i)\big)^{-1} \, \in Q(F)^*/ N(Q(F))^* \]
 for some $d\in \det(\alpha(\pi_1(S^3\setminus \nu L)))$ and $f\in F$.
\end{proposition}

If $\alpha$ is the trivial representation over $\mathbb{Z}$, then it is shown in \cite[Theorem~5.1.1]{Tu86} that the torsion
is represented by the untwisted Alexander polynomial of $L$ corresponding to $\psi$.
Using this observation we see that Proposition \ref{prop:slicelink} generalizes earlier
results of Murasugi \cite{Mu67}, Kawauchi \cite{Ka77,Ka78} and Nakagawa \cite{Na78}. The approach taken in \cite{CF10}
is partly inspired by  Turaev's proof of the untwisted case (cf. \cite[Theorem~5.4.2]{Tu86}).

In \cite{CF10} we will in particular use the obstruction of Proposition \ref{prop:slicelink} to reprove that the Bing double of the Figure 8 knot is not slice.
This had first been shown by Cha \cite{Ch09}.

\section{Twisted Alexander polynomials, the Thurston norm and fibered manifolds}  \label{section:fibgenus}

\subsection{Twisted Alexander polynomials and fibered manifolds}\label{section:fib}
Let $\phi\in H^1(N;\mathbb{Z}) = \mbox{Hom}(\pi_1(N),\mathbb{Z})$ be non--trivial. We say
\emph{$(N,\phi)$ fibers over $S^{1}$}
 if there exists
 a fibration $p:N\to S^1$ such that the induced map $p_*:\pi_1(N)\to \pi_1(S^1)=\mathbb{Z}$ coincides with $\phi$. If
$K$ is a fibered knot, then it is a classical result of Neuwirth that
$\Delta_K$ is monic and that $\deg(\Delta_K)=2 \, \mbox{genus}(K)$. Similarly, twisted Alexander polynomials provide necessary conditions to the fiberability of a pair $(N,\phi)$.

In order to state the fibering obstructions for a pair $(N,\phi)$
we need to introduce the Thurston norm. Given $(N,\phi)$ the \emph{Thurston norm} of $\phi$ is defined as
 \[
||\phi||_{T}=\min \{ \chi_-(S)\, | \, S \subset N \mbox{ properly embedded surface dual to }\phi\}.
\]
Here, given a surface $S$ with connected components $S_1\cup\dots \cup S_k$, we define
$\chi_-(S)=\sum_{i=1}^k \max\{-\chi(S_i),0\}$.
 Thurston
\cite{Th86} showed that this defines a seminorm on $H^1(N;\mathbb{Z})$ which can be extended to a seminorm
on $H^1(N;\mathbb{R} )$. As an example consider $S^3\setminus \nu K$, where $K\subset S^3$ is a non--trivial knot. Let $\phi\in H^1(S^3\setminus \nu K;\mathbb{Z})$ be a
generator, then  $||\phi||_T=2\, \mbox{genus}(K)-1$.

Let $N$ be a 3--manifold $N$ with empty or toroidal boundary and let  $\phi\in H^1(N;\mathbb{Z})$.
Recall that we identify the group ring $R[\mathbb{Z}]$ with the Laurent polynomial ring $R[t^{\pm 1}]$ and we will now identify $Q(R[\mathbb{Z}])$ with the field of rational functions $Q(t)$.
Given a  representation
${\gamma}:\pi_1(N)\to \mbox{GL}(k,R)$ we therefore view the corresponding twisted Reidemeister torsion
$\tau(N,{\gamma \otimes \phi})$ as an element in $Q(t)$ and we view the corresponding twisted
Alexander polynomials
$\Delta_{N,i}^{\gamma \otimes \phi}$ as elements in $R[t^{\pm 1}]$.

We say that the twisted Reidemeister torsion $\tau(N,\gamma\otimes \phi)$ is \emph{monic} if there exist
 polynomials
$p(t),q(t)\in R[t^{\pm 1}]$ with $\frac{p(t)}{q(t)}\doteq \tau(N,\gamma\otimes \phi)$ such that the
top coefficients of $p(t)$ and $q(t)$  lie in
\[ \{ \pm \det(\gamma(g)) \, |\, g\in \pi_1(N)\}.\]
We also say that the twisted Alexander polynomial $\Delta_{N,i}^{\gamma \otimes \phi}$ is \emph{monic}
if one (and equivalently any) representative has a top coefficient which is a unit in $R$.


We recall the following basic lemma:

\begin{lemma}
Let $N$ be a 3--manifold with empty or toroidal boundary. Let $\phi\in H^1(N;\mathbb{Z})$ be non--trivial  and let
${\gamma }:\pi_1(N)\to \mbox{GL}(k,R)$ be a  representation with $R$ a Noetherian UFD. Then the following hold:
\begin{enumerate}
\item $\Delta_{N,0}^{\gamma \otimes \phi}$ is monic,
\item If $\Delta_{N,1}^{\gamma \otimes \phi}$ is non--zero, then $\Delta_{N,2}^{\gamma \otimes \phi}$ is monic.
\end{enumerate}
\end{lemma}

\begin{proof}
As in Section \ref{section:comptau} we pick a cell decomposition of $N$ with one 0--cell $x_0$ and $n$ 1--cells $c_1,\dots,c_n$.
We denote the corresponding elements in $\pi_1(N,x_0)$ by $c_1,\dots,c_n$ as well.
Without loss of generality we can assume that $\phi(c_1)> 0$.
We then get a resolution for $H_0(N;R[t^{\pm 1}]^n)$ with presentation matrix
\[ \begin{pmatrix} \mbox{id}_n -({\gamma \otimes \phi})(c_1) &\dots & \mbox{id}_n -{(\gamma \otimes \phi)}(c_n)\end{pmatrix}. \]
We refer to \cite[Proof~of~Proposition~6.1]{FK06} for details.
We have
\[ \det(\mbox{id}_n-{(\gamma \otimes \phi)}(c_1))=\det(\mbox{id}_n-t^{\phi(c_1)}\gamma(c_1))\in R[t^{\pm 1}], \]
which is monic since the top coefficient equals $\det(\gamma(c_1))$. By definition
$\Delta_{N,0}^{\gamma \otimes \phi}$ divides $\det(\mbox{id}_n-({\gamma \otimes \phi})(c_1))$, we therefore see that $\Delta_{N,0}^{\gamma \otimes \phi}$ is monic.
The claim that $\Delta_{N,2}^{\gamma \otimes \phi}$ is monic now follows from Proposition \ref{prop:dualitydelta}.
\end{proof}

\begin{remark}
Note that if $\tau(N,{\gamma \otimes \phi})$ is monic, then it follows from the previous lemma and
from Proposition \ref{prop:taudelta} that $\Delta_{N,i}^{\gamma \otimes \phi}\in R[t^{\pm 1}]$ is monic for $i=0,1,2$.
Note though that the converse does not hold in general since  twisted Alexander polynomials have in general a higher indeterminacy than twisted Reidemeister torsion.
For example, let $\mathbb{F}$ be a field and let $\gamma:\pi_1(N)\to \mbox{GL}(k,\mathbb{F})$ be a representation such that  $\Delta_{N,i}^{\gamma\otimes \phi}\ne 0$. It follows immediately from the definition
that $\Delta_{N,i}^{\gamma\otimes \phi}$ is monic. However, $\tau(N,\gamma\otimes \psi)$
is not necessarily monic (cf. e.g. \cite[Example~4.2]{GKM05}).
\end{remark}

We can now formulate the following fibering obstruction theorem
which was proved in various levels of generality by
Cha \cite{Ch03}, Kitano and Morifuji \cite{KM05}, Goda, Kitano and Morifuji \cite{GKM05}, Pajitnov \cite{Pa07}, Kitayama \cite{Kiy08a} and  \cite{FK06}.

\begin{theorem} \label{thm:fibob}
Let $N$ be a 3--manifold. Let $\phi\in H^1(N;\mathbb{Z})$ be non--trivial such that  $(N,\phi)$ fibers over
$S^1$ and such that $N\ne S^1\times D^2, N\ne S^1\times S^2$. Let
${\gamma}:\pi_1(N)\to \mbox{GL}(k,R)$ be a  representation. Then  $\tau(N,{\gamma \otimes \phi}) \in Q(t)$ is monic
and we have
\[ k ||\phi||_T =\deg(\tau(N,{\gamma \otimes \phi})).\]
\end{theorem}

\begin{remark}
\begin{enumerate}
\item If $R$ is a Noetherian UFD, then the last equality can be rewritten as
\[ k ||\phi||_T =\deg \Delta_{N,1}^{{\gamma \otimes \phi}}-\deg \Delta_{N,0}^{{\gamma \otimes \phi}}-\deg \Delta_{N,2}^{{\gamma \otimes \phi}}.\]
\item Recall that an alternating knot is fibered if and only if its ordinary Alexander polynomial is monic. In contrast to this classical result it follows from calculations by Goda and Morifuji \cite{GM03} (cf. also \cite{Mo08})
that there exists an alternating knot such that a twisted Reidemeister torsion is monic, but which is not fibered.
\item Theorem \ref{thm:fibob} has been generalized by Silver and Williams \cite{SW09d} to give an obstruction for a general group to admit an epimorphism onto $\Z$ such that the kernel is a finitely generated free group.
\end{enumerate}
\end{remark}

\begin{proof}
The condition on the degrees is proved in \cite[Theorem~1.3]{FK06} for $R$ a Noetherian UFD.
The monicness of twisted Reidemeister torsion was proved by
Goda, Kitano and Morifuji \cite{GKM05} in the case of a knot complement.
The monicness of $\Delta_{N,1}^{\gamma \otimes \phi}$ was proved in \cite[Theorem~1.3]{FK06}.
The general case of Theorem \ref{thm:fibob}  can be obtained by a direct calculation as follows.
Let $S$ be the fiber and $f:S\to S$ the monodromy. We endow $S$ with a CW--structure such that $f$ is a cellular map.
Denote by $D_i$ the set of $i$--cells of $S$ and denote by $n_i$ the number of $i$--cells.
We can then endow $N=(S\times [0,1])/(x,0)\sim (f(x),1)$ with a CW--structure where the $i$--cells are given by
$D_i$ and $E_i\mathrel{\mathop:}=\{ c\times (0,1) \, | \, c\in D_{i-1}\}$. A direct calculation using \cite[Theorem~2.2]{Tu01} now shows
that
 $\tau(N,{\gamma \otimes \phi}) \in Q(t)$ is monic
and that
\[ \deg(\tau(N,{\gamma \otimes \phi}))=-k\chi(S)=k ||\phi||_T.\]
\end{proof}

The calculations in \cite{Ch03}, \cite{GKM05} and \cite{FK06} gave evidence that
twisted Alexander polynomials are very successful at detecting non--fibered manifolds.
The results of Morifuji \cite[p.~452]{Mo08} also give evidence to the conjecture that the twisted Alexander polynomial corresponding to a `generic' representation
detect fiberedness.

Using a deep result of Agol \cite{Ag08} the authors proved in \cite{FV08c} (see also \cite{FV10} for an outline of the proof)
the following converse to Theorem \ref{thm:fibob}.

\begin{theorem} \label{thm:fv08} Let $N$ be a $3$--manifold with empty or toroidal boundary.
 Let $\phi \in H^{1}(N;\mathbb{Z})$ a nontrivial class. If for
any epimorphism $\gamma:\pi_1(N)\to G$ onto a finite group
the twisted Alexander polynomial $\Delta_{N}^{\gamma\otimes \phi}\in \mathbb{Z}[t^{\pm 1}]$ is monic
and  \[ k||\phi||_T=\deg(\tau(N,\gamma\otimes \phi))\] holds,
then $(N,\phi)$ fibers over $S^1$.
 \end{theorem}

\begin{remark}
Building on work of Taubes \cite{Ta94,Ta95}, Donaldson \cite{Do96} and Kronheimer \cite{Kr99} the authors
also show that Theorem \ref{thm:fv08} implies the following: If $N$ is a closed 3--manifold
and if $S^1\times N$ is symplectic, then $N$ fibers over $S^1$. This provides a converse to a theorem
of Thurston \cite{Th76}. We refer to \cite{FV06,FV08a,FV08b,FV08c} for details, and we refer to
Kutluhan--Taubes \cite{KT09}, Kronheimer--Mrowka \cite{KM08} and  Ni \cite{Ni08} for an alternative proof in the case that $b_1(N)=1$.
\end{remark}

It is natural to ask whether $(N,\phi)$ fibers if all twisted Alexander polynomials are monic.
An affirmative answer  would be of great interest in the study of symplectic structures of 4--manifolds with a free circle action (cf. \cite{FV07b}).
An equivalent question has also been raised as a conjecture by Goda and Pajitnov \cite[Conjecture~13.2]{GP05} in the study of Morse--Novikov numbers. We refer to
 \cite{GP05} and \cite{Pa07} for more information on the relationship between twisted Alexander polynomials, twisted Novikov homology and Morse--Novikov numbers.

Somewhat surprisingly, there is strong evidence to the following much weaker conjecture:
A pair $(N,\phi)$ fibers if all twisted Alexander polynomials are non--zero.
In fact the authors showed the following theorem
(cf. \cite[Theorem~1.3]{FV07a} and \cite[Theorem~1,~Proposition~4.6,~Corollary~5.6]{FV08b}).

\begin{theorem}\label{thm:zerodelta}
Let $N$ be a 3--manifold with empty or toroidal boundary and $\phi \in H^1(N;\mathbb{Z})$ non--trivial.
Suppose that $\Delta_{N}^{\gamma\otimes\phi}$ is non--zero for any epimorphism $\gamma:\pi_1(N)\to G$ onto a finite group.
Furthermore suppose that one of the following holds:
\begin{enumerate}
\item $N=S^3\setminus \nu K$ and $K$ is a genus one knot,
\item $||\phi||_T=0$,
\item $N$ is a graph manifold,
\item $\phi$ is dual to a  connected incompressible
surface $S$ such that $\pi_1(S)\subset \pi_1(N)$ is separable,
\end{enumerate}
then $(N,\phi)$ fibers over $S^1$.
\end{theorem}

Here we say that a subgroup $A$ of a group $\pi$ is \emph{separable} if for any $g\in \pi\setminus A$ there exists an epimorphism ${\gamma}:\pi\to G$ onto a finite group $G$ such that ${\gamma }(g)\not\in {\gamma }(A)$.
It is conjectured (cf. \cite{Th82}) that given a hyperbolic 3--manifold $N$ any finitely generated subgroup
$A\subset \pi_1(N)$ is separable. In particular, if Thurston's conjecture is true, then Condition (4) of Theorem
\ref{thm:zerodelta} is satisfied for any hyperbolic $N$.

The following theorem of Silver--Williams \cite{SW09b} (cf. also \cite{SW09a}) gives an interesting criterion for a knot to have vanishing twisted Alexander polynomial.

\begin{theorem}\label{thm:sw}
Let $K\subset S^3$ a knot.
Then there exists an epimorphism ${\gamma }:\pi_1(S^3\setminus \nu K)\to G$ to a finite group such that $\Delta_{K}^{\gamma} =0$
if and only if the universal abelian cover of $S^3\setminus \nu K$ has uncountably many finite covers.
\end{theorem}

\subsection{Twisted Alexander polynomials and the Thurston norm}\label{section:genus}

It is a classical result of Alexander that given a knot $K\subset S^3$ the following inequality holds:
 \[ 2\mbox{genus}(K) \geq \deg(\Delta_K).  \]
This result was first generalized to general 3--manifolds by McMullen \cite{McM02}
and then to twisted Alexander polynomials in \cite{FK06}.
The following theorem is \cite[Theorem~1.1]{FK06}.
The proof builds partly on ideas of Turaev's in \cite{Tu02b}.

\begin{theorem} \label{thm:lowerboundphi} Let $N$ be a 3--manifold whose boundary is empty or consists of tori.
Let $\phi \in H^1(N;\mathbb{Z})$ be non--trivial and let $\gamma:\pi_1(N)\to \mbox{GL}(k,R)$ be a representation such that
$\Delta_{N}^{\gamma \otimes \phi}\ne 0$.
 Then
\[  ||\phi||_T \geq  \frac{1}{k} \deg(\tau(N,\gamma\otimes \phi)).\]
Equivalently,
\[ ||\phi||_T \geq \frac{1}{k}\big(\deg(\Delta_{N}^{\gamma\otimes \phi})- \deg(\Delta_{N,0}^{\gamma\otimes \phi}) -\deg(\Delta_{N,2}^{\gamma\otimes \phi})\big).\]
\end{theorem}

\begin{remark}
\begin{enumerate} In \cite{FK06} it was furthermore shown using `KnotTwister' (cf. \cite{Fr09}) that
twisted Alexander polynomials  detect the genus of all knots with up to twelve crossings.
\item
It seems reasonable to conjecture that given an irreducible 3--manifold $N$
twisted Alexander polynomials detect the Thurston norm for any $\phi \in H^1(N;\mathbb{Z})$.
A positive answer would have interesting consequences for 4--manifold topology as pointed out in
\cite{FV09}.
\item If $\Delta_{N}^{\gamma \otimes \phi}=0$, then we define the \emph{torsion twisted Alexander polynomial} $\tilde{\Delta}_N^{\gamma\otimes \phi}$ to be
the order of the $R[t^{\pm 1}]$--module
\begin{eqnarray*}
 &&\mbox{Tor}_{R[t^{\pm1}]}(H_1(N;R[t^{\pm1}]^k))=\\
&&=\{ v\in H_1(N;R[t^{\pm1}]^k)\, |\, \lambda v=0 \mbox{ for some }\lambda \in R[t^{\pm1}] \setminus \{0\}\,\}.
\end{eqnarray*}
It is then  shown in  \cite[Section~4]{FK06} that
the $\tilde{\Delta}^{\gamma\otimes \phi}_N$ also gives rise to give lower bounds on the Thurston norm.
We point out that  by \cite[Theorem~3.12~(3)]{Hi02} and  \cite[Lemma~4.9]{Tu01} the torsion twisted Alexander polynomial can be computed directly from a presentation
of $H_1(N;R[t^{\pm1}]^k)$.
\end{enumerate}
\end{remark}

Note that Theorem \ref{thm:lowerboundphi} gives lower bounds on the Thurston norm for a given $\phi\in H^1(N;\mathbb{Z})$.
In order to give bounds for the whole Thurston norm ball at once, we will introduce
 twisted Alexander norms, generalizing McMullen's Alexander norm \cite{McM02} and Turaev's torsion norm \cite{Tu02a}.
In the following let $N$ be a 3--manifold with empty or toroidal boundary such that $b_1(N)>1$.
Let  $\psi:\pi_1(N)\to F\mathrel{\mathop:}=H_1(N;\mathbb{Z})/\mbox{torsion}$ be the canonical projection map.
Let $\gamma:\pi_1(N)\to \mbox{GL}(k,R)$ be a  representation. If $\Delta_{N}^{\gamma\otimes \psi}=0$ then we set
$||\phi||_{A}^{\gamma}=0$ for all $\phi\in H^1(N;\mathbb{R} )$. Otherwise we write $\Delta_{N}^{\gamma\otimes \psi}=\sum a_if_i$
for $a_i\in R$ and $f_i \in F$. Given $\phi \in H^1(N;\mathbb{R} )$ we then define
\[ ||\phi||_{A}^{\gamma} \mathrel{\mathop:}=\max\{  \phi(f_i-f_j)\, |\, (f_i, f_j) \mbox{ such that }a_ia_j\ne 0\}.\]
Note that this norm is independent of the choice of representative of $\Delta_{N}^{\gamma\otimes \psi}$.
Clearly this defines a seminorm on $H^1(N;\mathbb{R} )$ which we call the
\emph{twisted Alexander norm of $(N,\gamma)$}. Note that if $\gamma :
\pi_1(N) \to \mbox{GL} (1,\mathbb{Z})$ is the trivial representation, then we just obtain McMullen's Alexander norm.

The following is proved in \cite[Theorem~3.1]{FK08a},
but we also refer to the work of McMullen \cite{McM02}, Turaev \cite{Tu02a}, \cite[Section~6]{Tu02b}, Harvey \cite{Ha05} and Vidussi \cite{Vi99,Vi03}.
The main idea of the proof is to combine  Theorem \ref{thm:lowerboundphi} with Corollary \ref{cor:changevardelta}.

\begin{theorem}
Let $N$ be a 3--manifold with empty or toroidal boundary such that $b_1(N)>1$. Let $\psi:\pi_1(N)\to F\mathrel{\mathop:}=H_1(N;\mathbb{Z})/\mbox{torsion}$ be the canonical projection map
and let $\gamma:\pi_1(N)\to \mbox{GL}(k,R)$ be a representation with $R$ a Noetherian UFD.
Then for any $\phi\in H^1(N;\mathbb{R} )$ we have
\[ ||\phi||_T \geq \frac{1}{k}||\phi||_A^\gamma\]
and equality holds for any $\phi$ in a fibered cone of the Thurston norm ball.
\end{theorem}

We refer to \cite{McM02} and \cite{FK08a} for some calculations, and we refer to \cite{Du01} for more information on the relationship
between the Alexander norm and the Thurston norm.

\subsection{Normalized twisted Reidemeister torsion and the free genus of a knot}\label{section:normalized}

Let $K\subset S^3$ a knot and let $\gamma:\pi_1(S^3\sm \nu K)\to \mbox{GL}(k,R)$ be a representation with $R$ a Noetherian UFD.
Let $\epsilon=\det(\gamma(\mu))$ where $\mu$ denotes a meridian of $K$.
Kitayama \cite{Kiy08a} introduced an invariant $\tilde{\Delta}_{K,\gamma}\in Q(R(\epsilon^{\frac{1}{2}})[t^{\pm \frac{1}{2}}])$ which is a normalized version of the twisted Reidemeister torsion $\tau(K,\gamma)$, i.e.  $\tau(K,\gamma)$ has no indeterminacy and up to multiplication by an element of the form $\epsilon^r t^{\frac{s}{2}}$ it is a representative of the twisted Reidemeister torsion. (The invariant $\tilde{\Delta}_{K,\gamma}$ should not be confused with the torsion polynomial introduced in Section \ref{section:linkslice}.)

Kitayama \cite[Theorem~6.3]{Kiy08a} studies the invariant $\tilde{\Delta}_{K,\gamma}$ for fibered knots, obtaining a refined version of the fibering obstruction which we stated in Theorem \ref{thm:fibob}. Furthermore \cite[Theorem~5.8]{Kiy08a} proves a duality theorem for $\tilde{\Delta}_{K,\gamma}$, refining  Proposition \ref{prop:dualitytau}.

In the following we say that $S$ is a \emph{free Seifert surface} for $K$ if $\pi_1(S^3\setminus S)$ is a free group.
Note that Seifert's algorithm produces free Seifert surfaces, in particular any knot has a free Seifert surface.
Given $K$ the \emph{free genus} is now defined as
\[ \mbox{free-genus}(K)=\mbox{min}\{ \mbox{genus}(S) \, |\, \mbox{$S$ free Seifert surface for $K$}\}.\]
Clearly we have $\mbox{free-genus}(K)\geq \mbox{genus}(K)$.
In order to state the lower bound on the free genus coming from $\tilde{\Delta}_{K,\gamma}$ we have to make a few more definitions.
Given a Laurent polynomial $p=\sum_{i=k}^l a_it^i\in R[t^{\pm \frac{1}{2}}]$ with $a_k\ne 0, a_l\ne 0$
we write $\mbox{l-deg}(p)=k$ (`lowest degree') and $\mbox{h-deg}(p)=l$ (`highest degree').
Furthermore given $f\in Q(R[t^{\pm \frac{1}{2}}])$ we  define
\[ \mbox{h-deg}(f)=\mbox{h-deg}(p)-\mbox{h-deg}(q)\]
where we pick $p,q\in R[t^{\pm \frac{1}{2}}]$ with $f=\frac{p}{q}$.
Kitayama \cite[Proposition~6.6]{Kiy08a} proved the following theorem:

\begin{theorem}
Given a knot $K\subset S^3$ and a representation $\gamma:\pi_1(S^3\sm \nu K)\to \mbox{GL}(k,R)$  with $R$ a Noetherian UFD
we have
\[ 2k  \mbox{free-genus}(K) \geq 2\,\mbox{h-deg}(\tilde{\Delta}_{K,\gamma})+k.\]
\end{theorem}

Note that if $\mbox{h-deg}(\tilde{\Delta}_{K,\gamma})=-\, \mbox{l-deg}(\tilde{\Delta}_{K,\gamma})$ (which is the case for unitary representations)
then the bound in the theorem is implied by the bound given in Theorem \ref{thm:lowerboundphi}.
It is a very interesting question whether Kitayama's bound can detect the difference between the genus and the free genus.

Recall that Lin's original definition of the twisted Alexander polynomial is in terms of regular Seifert surfaces,
it might be worthwhile to explore the possibility that an appropriate version of Lin's twisted Alexander polynomial gives lower bounds on the `regular genus' of a knot
which can tell the regular genus apart from the free genus.

\section{Twisted invariants of knots and special representations}\label{section:twodim}

Given a knot $K\subset S^3$ the representations which have been studied most are the 2--dimensional complex representations  and the metabelian representations.
It is therefore natural to consider special properties of twisted Alexander polynomials corresponding to such  representations.



\subsection{Parabolic representations}
\label{section:specialknots}

Let $K\subset S^3$ be a knot. A representation
$\gamma:\pi_1(S^3\setminus \nu K)\to \mbox{SL}(2,\mathbb{C})$ is called \emph{parabolic}
if the
image of any meridian is a matrix with trace 2. Note that Thurston \cite{Th87} showed that
given a hyperbolic knot $K$   the discrete faithful representation
$\pi_1(S^3\setminus \nu K) \to \mbox{PSL}(2,\mathbb{C})$ lifts to a
parabolic representation $\gamma : \pi_1(S^3\setminus \nu K) \to \mbox{SL}(2,\mathbb{C})$.
The twisted Reidemeister torsion corresponding to this canonical representation has been surprisingly little studied
(cf. though \cite{Sug07} and \cite[Corollary~4.2]{Mo08}).

Throughout the remainder of this section let $K$ be a 2--bridge knot. Then the group $\pi_1(S^3\setminus \nu K)$ has a presentation of the form
$\langle x,y | Wx=yW \rangle$ where $x,y$ are meridians of $K$ and $W$ is a word in $x^{\pm 1}, y^{\pm 1}$.
Parabolic representations of 2--bridge knots have been extensively studied by Riley \cite[Section~3]{Ri72}.
To a 2--bridge knot $K$ Riley associates a monic polynomial $\Phi_K(t)\in \mathbb{Z}[t^{\pm 1}]$
such that  any zero $\zeta$ of $\Phi_K(t)$ gives rise to a representation
$\gamma_\zeta:\pi_1(S^3\setminus \nu K)\to \mbox{SL}(2,\mathbb{C})$ of the form
\[ \gamma_\zeta(x)= \begin{pmatrix} 1&1\\ 0&1 \end{pmatrix} \mbox{ and }\gamma_\zeta(y)= \begin{pmatrix} 1&0\\ \zeta &1 \end{pmatrix}. \]
Furthermore Riley shows that any parabolic representation of a 2--bridge knot is conjugate to such a representation
(cf. also \cite[Section~5]{SW09c} for details).

Given an irreducible factor $\phi(t)$ of $\Phi_K(t)$ of degree $d$ one can consider the representation $\oplus \gamma_{\zeta'}$ where $\zeta'$ runs over the set of all zeroes of $\phi(t)$.
Silver and Williams show that this representation is conjugate to an integral representation $\gamma_{\phi(t)}:\pi_1(S^3\setminus \nu K)\to \mbox{GL}(2d,\mathbb{Z})$
which is called the \emph{total representation corresponding to $\phi(t)$}.

Twisted Alexander polynomials of 2--bridge knots corresponding to such parabolic and total representations have been studied
extensively by Silver and Williams \cite{SW09c} and Hirasawa and Murasugi (cf. \cite{Mu06} and \cite{HM08}).
In particular the following theorem is shown. It should be compared  to the classical result that given a knot
$K$ we have $\Delta_K(1)=1$.

\begin{theorem}
Let $K$ be a two--bridge knot and $\phi(t)$ an irreducible factor of $\Phi_K(t)$ of degree $d$.
Then
\[ \left|\Delta_K^{\gamma_{\phi(t)}}(1)\right| =2^d.\]
\end{theorem}

A special case of the
theorem
is shown in \cite[Theorem~A]{HM08},
the general case is proved in \cite[Theorem~6.1]{SW09c}.
Silver and Williams also conjectured that under the assumptions of the theorem we have
\[ \left|\Delta_K^{\gamma_{\phi(t)}}(-1)\right| =2^d m^2\]
for some odd number $m$.

We refer to \cite{SW09c} and \cite{HM08} for more on twisted Alexander polynomials of 2--bridge knots.
All three  papers contain a wealth of interesting examples and results which we find impossible to summarize in this short survey.

Silver and Williams also considered twisted Alexander polynomials of torus knots.
They showed \cite[Section~7]{SW09c} that for any parabolic representation $\gamma$ of a torus knot $K$, the
twisted Alexander polynomial $\Delta^\gamma_K$ is a product of cyclotomic polynomials.

\subsection{Twisted Alexander polynomials and the space of 2--dimensional representations}

It is a natural question to study the behavior of twisted Alexander polynomials under a change of representation.
Recall that given a group $\pi$ we say that two representations $\alpha,\beta:\pi\to \mbox{GL}(k,R)$ are
\emph{conjugate} if there exists $P\in \mbox{GL}(k,R)$ with $\alpha(g)=P\beta(g)P^{-1}$ for all $g\in \pi$.
By Lemma \ref{lem:equ:rep} we can view  Alexander polynomials as function on the set of conjugacy classes of representations.

Let $K$ be a 2--bridge knot. Riley \cite{Ri84}
(cf. also \cite[Proposition~3]{DHY09}) showed that
conjugacy classes of representations $\pi_1(S^3\setminus \nu K)\to \mbox{SL}(2,\mathbb{C})$
correspond to the zeros of an affine algebraic curve in $\mathbb{C}^2$.
The twisted Reidemeister torsion corresponding to these representations for twist knots have been studied by
Morifuji \cite{Mo08} (cf. also \cite{GM03}). The computations show in particular that for twist knots
twisted Reidemeister torsion detects fiberedness and the genus for all but finitely many conjugacy classes of non--abelian $\mbox{SL}(2,\mathbb{C})$--representations.

Regarding twist knots we also refer to the work of
Huynh and Le \cite[Theorem~3.3]{HL07} who found an unexpected  relationship between twisted Alexander polynomials and the A--polynomial.
(Note thought that their definition of twisted Alexander polynomials differs somewhat from our approach.)

Finally we refer to Kitayama's work \cite{Kiy08b} for certain symmetries when we view twisted Alexander polynomials of knots in rational homology spheres
as a function on the space of 2--dimensional regular unitary representations.

%
%
%
%

\subsection{Twisted invariants of hyperbolic knots and links}

Let $L\subset S^3$ be a hyperbolic link. Then the corresponding representation $\pi_1(S^3\sm \nu L)\to PSL(2,\C)$ lifts
to a canonical representation $\g_{can}:\pi_1(S^3\sm \nu L)\to SL(2,\C)$ and
it induces the adjoint representation $\g_{adj}: \pi_1(S^3\sm \nu L)\to PSL(2,\C)\to \aut(sl(2,\C))\cong \mbox{SL}(3,\C)$.

The corresponding invariant $\tau(L,\g_{adj})$ has been studied in great detail by Dubois and Yamaguchi \cite{DY09}.
In particular it is shown that $\tau(L,\g_{adj})$ is a symmetric non-zero polynomial (the non-vanishing result builds on work
of Porti \cite{Po97}). Furthermore the invariant $\tau(L,\g_{adj})$ is computed explicitly for many examples.

In \cite{DFJ10} it is shown that given a knot $K$ the invariant $\tau(K,\g_{can})$ gives rise to a well-defined  non-zero
invariant $\tkt$.
(This result builds on work of Menal-Ferrer and Porti \cite{MP09}).
This invariant is computed for all knots up to 13 crossings, for all these knots the invariant $\tkt$ detects
the genus, the fiberedness and the chirality.

\subsection{Metabelian representations}

Given a group $G$ the
derived series of $G$ is defined  inductively  by $G^{(0)}=G$ and $G^{(i+1)}=[G^{(i)},G^{(i)}]$.
A representation of a group $G$ is called metabelian if it factors through $G/G^{(2)}$.
Note that if $K$ is a knot, then the metabelian quotient $\pi_1(S^3\setminus \nu K)/\pi_1(S^3\setminus \nu K)^{(2)}$ is well--known to be determined by the Alexander module of $K$
(cf. e.g. \cite[Section~2]{BF08}).
Metabelian representations and metabelian quotients of knot groups have been studied extensively, we refer to \cite{Fo62,Fo70}, \cite{Hat79}, \cite{Fr04},  \cite{Je08} and \cite{BF08}
for more information.

The structure of twisted Alexander polynomials of knots corresponding to certain types of metabelian representations has been studied in detail by Hirasawa and Murasugi
\cite{HM09a,HM09b}. These papers contain several interesting conjectures regarding special properties of such twisted Alexander polynomials, furthermore these conjectures
are verified for certain classes of 2--bridge knots and further evidence is given by explicit calculations.

\section{Miscellaneous applications of twisted Reidemeister torsion to knot theory}\label{section:misc}

In this section we will summarize various  applications of twisted Alexander polynomials to the study of knots and links.

\subsection{A Partial order on knots}\label{section:po}

Given a knot $K$ we write $\pi_1(S^3\sm \nu K)\mathrel{\mathop:}=\pi_1(S^3\setminus K)$.
 For two prime
knots $K_1$ and $K_2$ one defines $K_1 \geq K_2$ if there exists a surjective group homomorphism
$\varphi:\pi_1(K_1)\to \pi_1(K_2)$.
The relation $``\geq"$ defines a partial order on the set of prime knots (cf. \cite{KS05a}). Its study is often related to a well--known conjecture of J. Simon, that posits that for a given $K_1$, the set of knots $K_2$ s.t. $K_1 \geq K_2$ is finite.

Kitano, Suzuki and Wada \cite{KSW05} prove the following theorem
which generalizes a result of Murasugi \cite{Mu03}.

\begin{theorem}\label{thm:epi}
Let $K_1$ and $K_2$ be two knots in $S^3$ and let $\varphi:\pi_1(K_1)\to \pi_1(K_2)$ an epimorphism.
Let $\gamma:\pi_1(K_2)\to \mbox{GL}(k,R)$ a representation with $R$ a Noetherian UFD. Then
\[ \frac{\tau(K_1,\gamma\circ \varphi)}{\tau(K_2,\gamma )} \]
is an element in $R[t^{\pm 1}]$.
\end{theorem}

This theorem plays a crucial role in determining the partial order on the set of knots with up to eleven crossings.
We refer to \cite{KS05a}, \cite{KS05b}, \cite{KS08} and \cite{HKMS09} for details.

\subsection{Periodic and freely periodic knots}\label{section:periodic}

A knot $K\subset S^3$ is called \emph{periodic of period $q$}, if there exists a smooth transformation of $S^3$ of order $q$ which leaves $K$ invariant and such that the fixed point set is a circle $A$ disjoint from $K$.
Note that $A\subset S^3$ is the trivial knot by the Smith conjecture. We refer to \cite[Section~10.1]{Ka96} for more details on periodic knots.

Now assume that $K\subset S^3$ is a periodic knot of period $q$ with $A$ the fixed point set of $f:S^3\to S^3$.
Note that $S^3/f$ is diffeomorphic to $S^3$. We denote by $\pi$ the projection map $S^3\to S^3/f=S^3$ and we  write $\overline{K}=\pi(K), \overline{A}=\pi(A)$.

The following two theorems of Hillman, Livingston and Naik \cite[Theorems~3~and~4]{HLN06}
generalize  results of Trotter \cite{Tr61} and Murasugi \cite{Mu71}.

\begin{theorem}
Let $K$ be a periodic knot of period $q$. Let $\pi, A, \overline{A}$ and $\overline{K}$ as above.
Let $\overline{\gamma}:\pi_1(S^3\setminus \nu \overline{K})\to \mbox{GL}(n,R)$ a representation with $R=\mathbb{Z}$ or $R=\mathbb{Q}$.  Write $\gamma =\overline{\gamma}\circ \pi_*$.
Then there exists a polynomial $F(t,s)\in R[t^{\pm 1}, s^{\pm 1}]$ such that
\[ \Delta_K^{\gamma}(t) \doteq \Delta_{\overline{K}}^{\overline{\gamma}}(t) \prod_{k=1}^{q-1} F(t,e^{2\pi ik/q}).\]
\end{theorem}

In the untwisted case the polynomial $F(t,s)$ is just the Alexander polynomial of the ordered link $\overline{K}\cup \overline{A}\subset S^3$. We refer to \cite[Section~6]{HLN06} for more information on $F(t,s)$.
The following theorem gives an often stronger condition when the period is a prime power.

\begin{theorem}
Let $K$ be a periodic knot of period $q=p^r$ where $p$ is a prime. Let $\pi, A, \overline{A}$ and $\overline{K}$ as above.
Let $\overline{\gamma}:\pi_1(S^3\setminus \nu \overline{K})\to \mbox{GL}(n,\mathbb{Z}_p)$ be a representation. Write $\gamma =\overline{\gamma}\circ \pi_*$. If $\Delta_{K}^\gamma(t)\ne 0$, then
\[ \Delta_{K}^\gamma(t) \cdot \left(\Delta_{K,0}^\gamma(t)\right)^{q-1} \doteq \Delta_{\overline{K}}^{\overline{\gamma}}(t)^q\,
\left(\det\left(\mbox{id}_n-\gamma(A)t^{lk(K,A)}\right)\right)^{q-1}\, \in \mathbb{Z}_p[t^{\pm 1}],  \]
where we view $A$ as an element in $\pi_1(S^3\setminus \nu K)$.
\end{theorem}

(Note that Elliot \cite{El08} gave an alternative proof for this theorem.)
These theorems are applied in \cite[Section~10]{HLN06} to give obstructions on the periodicity of the Kinoshita--Terasaka knot and the Conway knot. Note that both knots have trivial Alexander polynomial, in particular Murasugi's obstructions are satisfied trivially.

A knot $K\subset S^3$ is called \emph{freely periodic of period $q$} if there exists a free transformation $f$ of $S^3$ of order $q$ which leaves $K$ invariant. We refer to \cite[Section~10.2]{Ka96} for more information on freely periodic knots. Given  such a freely periodic knot $K$ we denote by $\pi$ the projection map $S^3\to \Sigma\mathrel{\mathop:}=S^3/f$ and we  write $\overline{K}=\pi(K)$.

The following theorem is \cite[Theorem~5]{HLN06}; the untwisted case was first proved by
 Hartley \cite{Hat81}.

\begin{theorem}
Let $K$ be a freely periodic knot of period $q$.
Let $\pi, \S$  and $\overline{K}$ as above.
Let $\overline{\gamma}:\pi_1(\S \setminus \overline{K})\to \mbox{GL}(k,R)$ be a representation with $R$ a Noetherian UFD.
Write $\gamma =\overline{\gamma}\circ \pi_*$.
Then
\[ \Delta_{K}^\gamma(t^q) \doteq \prod_{k=0}^{q-1} \Delta_{\overline{K}}^{\overline{\gamma}}(e^{2\pi ik/q}t).\]
\end{theorem}

We refer to \cite[Section~11]{HLN06} for an application of this theorem to a case which could not be settled with Hartley's theorem.

\subsection{Zeroes of twisted Alexander polynomials and non--abelian representations}

Let $N$ be a 3--manifold with one boundary component and $b_1(N)=1$.
Put differently, let  $N$ be the complement of a knot in a rational homology sphere.
Let $\alpha:\pi_1(N)\to \mathbb{C}^*=\mbox{GL}(1,\mathbb{C})$ be a one--dimensional representation.
Note that $\alpha$ necessarily factors through a representation
$H_1(N;\mathbb{Z})\to \mbox{GL}(1,\mathbb{C})$ which we also denote by $\alpha$.
Heusener and Porti \cite{HP05} ask when the abelian representation
\[ \begin{array}{rcl} \rho_{\alpha}: \pi_1(N) &\to & \mbox{PSL}(2,\mathbb{C}) \\
 g&\mapsto &\pm \begin{pmatrix} \alpha(g)^{1/2} &0 \\ 0&\alpha(g)^{-1/2}\end{pmatrix}\end{array} \]
can be deformed into an irreducible representation.

 Denote by $\phi\in \mbox{Hom}(\pi_1(N),\mathbb{Z})=H^1(N;\mathbb{Z})\cong \mathbb{Z}$ a generator.
Pick $\mu \in H_1(N;\mathbb{Z})$ with $\phi(\mu)=1$.
Denote by $\sigma(\alpha,\mu)$ the representation
\[ \begin{array}{rcl} \pi_1(N)&\to & \mbox{GL}(1,\mathbb{C})\\
g &\mapsto & \alpha(g\mu^{-\phi(g)}).\end{array} \]

Heusener and Porti then give necessary and sufficient conditions for $\rho_\alpha$
to be deformable into an irreducible representation in terms of
the order of vanishing of $\Delta_{N}^{\sigma(\alpha,\mu)\otimes \phi}(t)\in \mathbb{C}[t^{\pm 1}]$ at $\alpha(\mu)\in \mathbb{C}^*$.
We refer to \cite[Theorems~1.2~and~1.3]{HP05} for more precise formulations and more detailed results.

Note that this result is somewhat reminiscent of the earlier results of Burde \cite{Bu67} and de Rham \cite{dRh68}  who showed that zeroes of the (untwisted) Alexander polynomial give rise to metabelian representations of the knot group.
We also refer to \cite{SW09e} for another relationship between zeros of twisted Alexander polynomials and the representation theory of knot groups.

\subsection{Seifert fibered surgeries}\label{section:sfs}

In \cite[Section~3]{Kiy09} Kitayama gives a surgery formula for twisted Reidemeister torsion.
Furthermore in \cite[Lemma~4.3]{Kiy09} a formula for the Reidemeister torsion of a Seifert fibered space is given.
By studying a suitable invariant derived from twisted Reidemeister torsion
an obstruction for a Dehn surgery on a knot to equal a specific Seifert fibered space are given.
These obstructions generalize Kadokami's obstructions given in \cite{Ka06} and \cite{Ka07}.
Finally Kitayama applies these methods to show  that
there exists no Dehn surgery on the Kinoshita--Terasaka knot which is homeomorphic
to any Seifert fibered space of the form $ M(p_1/q_1,p_2/q_2,p_3/q_3)$ (we refer to \cite[Section~4]{Kiy09} for the notation).

\subsection{Homology of cyclic covers}\label{section:cycliccovers}

Let $K\subset S^3$ be a knot. We denote by $H\mathrel{\mathop:}=H_1(S^3\setminus \nu K;\mathbb{Z}[t^{\pm 1}])$ its Alexander module. Given $n\in \mathbb{N}$ we denote by $L_n$ the  $n$--fold cyclic branched cover of $K$.
Recall that we have a canonical isomorphism $H/(t^n-1)\cong H_1(L_n)$.
Put differently, the Alexander module determines the homology of the branched covers. In the same vein,  the following formula due to Fox (\cite{Fo56} and see also \cite{Go78})
shows  that the Alexander polynomial determines the size of the homology of a branched cover:
\begin{equation}
| H_1(L_n )|=\left| \prod_{j=1}^{n-1} \Delta_K(e^{2\pi ij/n})\right|.
\end{equation}
Here we write $|H_1(L_n)|=0$ if $H_1(L_n)$ has positive rank. Gordon \cite{Go72} used this formula to study extensively the homology of the branched covers of a knot.
Gordon \cite[p.~366]{Go72} asked whether the non--zero values of the sequence $|H_1(L_n)|$ converge to infinity
if there exists a zero of the Alexander polynomial which is not a root of unity.

Given a multivariable polynomial $p\mathrel{\mathop:}=p(t_1,\dots,t_n)\in \mathbb{C}[t_1^{\pm 1},\dots,t_n^{\pm 1}]$ the Mahler measure is defined as
\[ m(p)=\exp \, \int_{\theta_1=0}^{1} \dots  \int_{\theta_n=0}^{1}\mbox{log} \left|p\big(e^{2\pi i \theta_1},\dots,e^{2\pi i \theta_1}\big) \right| \,d \theta_1 \dots d \theta_n.\]
Note that the integral  can be singular, but one can show that the integral always
converges. It is known (cf. e.g. \cite{SW02}) that an integral one variable polynomial $p$ always satisfies $m(p)\geq 1$, and it satisfies
$m(p)=1$ if and only if all zeroes of $p$ are roots of unity.

\begin{theorem}
Let $K$ be any knot, then
\[ \lim_{n\to \infty} \frac{1}{n}\mbox{log} |\mbox{Tor} \,H_1(L_n)|  = \log(m(\Delta_K(t))).\]
\end{theorem}

This theorem was proved for most cases by Gonz\'alez-Acu\~na and Short \cite{GS91}, the most general statement was proved by Silver and Williams \cite[Theorem~2.1]{SW02}. We also refer to \cite{Ri90} for a related result.
Note that by the above discussion this theorem in particular implies the affirmative answer to Gordon's question.

Silver and Williams also generalized this theorem to links (\cite[Theorem~2.1]{SW02}), relating the Mahler measure multivariable Alexander polynomial to the
homology growth of finite abelian covers of the link.
Finally in \cite[Section~3]{SW09c} these results are extended to the twisted case for certain representations
(e.g. integral representations). We refer to \cite[Section~3]{SW09c} for the precise formulations and to \cite[Section~5]{SW09c} for an interesting example.

\subsection{Alexander polynomials for links in $\mathbb{R}  P^3$}\label{section:rp}

Given a link $L\subset \mathbb{R}  P^3$ Huynh and Le \cite[Section~5.3.2]{HL08} use Reidemeister torsion corresponding to abelian representations
to define an invariant $\nabla_L(t)$ which lies in general in $\mathbb{Z}[t^{\pm 1}, (t-t^{-1})^{-1}]$ and has no indeterminacy.
Furthermore they show in \cite[Theorem~5.7]{HL08} the surprising fact that this invariant satisfies in fact a skein relation.

\section{Twisted Alexander polynomials of CW--complexes and groups}\label{section:general}

\subsection{Definitions and basic properties} \label{section:generalprop}

Let $X$ be a  CW--complex  with finitely many cells in each dimension.
Assume we are given a  non--trivial homomorphism $\psi:\pi_1(X)\to F$ to a torsion--free abelian group and a  representation
$\gamma:\pi_1(X)\to \mbox{GL}(k,R)$ where $R$ is a Noetherian UFD.
As in Section \ref{section:twialex} we can define the twisted Alexander modules
\[ H_i(X;R[F]^k)=H_i(C_*(\tilde{X};\mathbb{Z})\otimes_{\mathbb{Z}[\pi_1(X)]} R[F]^k)\]
where $\pi_1(X)$ acts on $C_*(\tilde{X};\mathbb{Z})$ by deck transformations and on $R[F]^k$ by $\gamma\otimes \psi$.
These modules are finitely presented and we can therefore define the twisted Alexander polynomial $\Delta^{\gamma\otimes \psi}_{X,i}\in R[F]$
to be the order of the $R[F]$--module $H_i(X;R[F]^k)$.
Note that twisted Alexander polynomials are homotopy invariants, in  particular given any manifold homotopy equivalent to a finite CW--complex we can define the twisted Alexander polynomials $\Delta^{\gamma\otimes \psi}_{X,i}\in R[F]$.

Let $G$ be a finitely presented group and $X=K(G,1)$ its  Eilenberg--Maclane space.
Given  a  non--trivial homomorphism $\psi:G\to F$ to a torsion--free abelian group and a  representation
$\gamma:G\to \mbox{GL}(k,R)$ where $R$ is a Noetherian UFD we define
\[ \Delta^{\gamma\otimes \psi}_{G,i}=\Delta^{\gamma\otimes \psi}_{K(G,1),i}.\]

\begin{remark}
\begin{enumerate}
\item For $i=0,1$ the Alexander polynomials $\Delta^{\gamma\otimes \psi}_{X,i}\in R[F]$ can be computed using Fox calculus
as in Section \ref{section:compdelta}.
\item Note that unless the Euler characteristic of $X$ vanishes we can not define the Reidemeister torsion corresponding to $(X,\psi,\gamma)$.
\item Most of the results of Section \ref{section:basics} do not hold in the general context.
The only results which do generalize are Proposition \ref{prop:taudelta} (2), Lemma \ref{lem:equ:rep} and Theorem \ref{thm:shapiro}.
\item
Note that given a finitely presented group $G$ and $\psi$, $\gamma$ as above Wada \cite{Wa94}
introduced an invariant which we refer to as $W(G,\gamma\otimes \psi)\in Q(R[F])$.
Using \cite[Lemma~4.11]{Tu01} one can  show that
\[ W(G,\gamma \otimes \psi) \doteq \frac{\Delta^{\gamma\otimes \psi}_{G,1} }{\Delta^{\gamma\otimes \psi}_{G,0}}.\]
In the literature Wada's invariant is often referred to as the  twisted Alexander polynomial of a group.
\end{enumerate}
\end{remark}

\subsection{Twisted Alexander polynomials of  groups}

The twisted Alexander polynomial  has been calculated  by Morifuji \cite[Theorem~1.1]{Mo01} for the braid groups $B_n$ with   $\psi:B_n\to \mathbb{Z}$ the abelianization map
and together with the Burau representation. Morifuji \cite[Theorem~1.2]{Mo01} also proves a symmetry theorem for twisted Alexander polynomials
of braid groups for  Jones representations corresponding to dual Young diagrams.

In \cite{Suz04} Suzuki shows that the twisted Alexander polynomial of the braid group $B_4$ corresponding to the abelianization $\phi:B_4\to \mathbb{Z}$ and the Lawrence--Krammer representation is trivial.
This shows in particular that the twisted Alexander polynomial of a group corresponding to a faithful representation can be trivial.
It is an interesting question whether given a knot and a faithful representation the twisted Alexander polynomial can ever be trivial.

Given a 2--complex $X$ Turaev \cite{Tu02c} introduced a norm on $H^1(X;\mathbb{R} )$ to which we refer to as the Turaev norm.
The definition of the Turaev norm is inspired
by the definition of the Thurston norm \cite{Th86}.
Turaev \cite{Tu02c} uses twisted Alexander polynomials of $X$ corresponding to one--dimensional representations
to define a twisted Alexander norm similar to the one defined in Section \ref{section:genus}. Turaev goes on to show that the twisted Alexander norm gives
a lower bound on the Turaev norm. We refer to \cite[Section~7.1]{Tu02b} for more information.

\subsection{Plane algebraic curves}\label{section:cf}

Let $\mathcal{C}\subset \mathbb{C}^2$  be an affine algebraic curve. Denote by $P_1,\dots,P_k$ the set of singularities and denote
by $L_1,\dots,L_k$ the links at the singularities and let $L_\infty$ be the link at infinity
(we refer to \cite{CF07} for details).
Note that $\mathbb{C}^2\setminus \nu \mathcal{C}$ is homotopy equivalent to a finite CW--complex.
By Section \ref{section:generalprop}  we can therefore consider the twisted Alexander polynomial of $\mathbb{C}^2\setminus \nu \mathcal{C}$.
Now let  $\gamma:\pi_1(\mathbb{C}^2\setminus \nu \mathcal{C})\to \mbox{GL}(k,\mathbb{F})$ be a representation where $\mathbb{F}\subset \mathbb{C}$ is a subring
closed under conjugation. Let $\phi: \mathbb{C}^2\setminus \nu \mathcal{C}\to \mathbb{Z}$ be the map given by sending each oriented meridian to one.
Cogolludo and Florens \cite[Theorem~1.1]{CF07} then  relate  twisted Alexander polynomial of
$\mathbb{C}^2\setminus \nu \mathcal{C}$ corresponding to $\gamma\otimes \phi$ to the one--variable twisted Alexander polynomials of the links $L_1,\dots,L_k$ and $L_\infty$.
This result generalizes a result of  Libgober's regarding untwisted Alexander polynomials of affine algebraic curves (cf.  \cite[Theorem~1]{Lib82}). We refer to \cite[Section~6]{CF07} for applications of this result. Finally we  refer to \cite{CS08} for a further application of twisted Alexander polynomials to algebraic geometry.

%
%
%
%
%

\section{Alexander polynomials and representations over non--commutative rings}\label{section:ho}

In the previous sections we only considered finite dimensional representations over commutative rings.
One possible approach to studying invariants corresponding to infinite dimensional representations
is to use the theory of $L^2$--invariants. We refer to \cite{Lu02} for the definition of various $L^2$--invariants  and for some applications to low--dimensional topology.
Even though $L^2$--invariants are a powerful tool they have not yet been systematically studied for links and 3--manifolds.
We refer to the work of Li and Zhang \cite{LZ06a,LZ06b} for some initial work.
We also would like to use this opportunity to advertise a problem stated in \cite[Section~3.2~Remark~(3)]{FLM09}.

For the remainder of this section we will now be concerned with invariants corresponding to finite dimensional representations
over non--commutative rings.
The study of such invariants (often referred to as higher order Alexander polynomials) was initiated by Cochran  \cite{Co04},
building on ideas of Cochran, Orr and Teichner \cite{COT03}. The notion of higher order Alexander polynomials
was extended to 3--manifolds by Harvey \cite{Ha05} and Turaev \cite{Tu02b}.
This theory is different in spirit to the fore mentioned $L^2$--invariants, but we refer to \cite{Ha08} and \cite[Proposition~2.4]{FLM09} for some connections.

\subsection{Non--commutative Alexander polynomials}\label{sec:phicompatible}

Let $\mathbb{K}$ be a (skew) field and $\gamma:\mathbb{K}\to \mathbb{K}$ a ring homomorphism. Denote by $\mathbb{K}[t^{\pm 1}]$ the corresponding skew Laurent
polynomial ring over $\mathbb{K}$. The elements in $\mathbb{K}[t^{\pm 1}]$ are formal sums
$\sum_{i=-r}^s a_it^i$ with $a_i\in \mathbb{K}$ and multiplication in $\mathbb{K}[t^{\pm 1}]$ is given by the rule
$t^ia=\gamma^i(a)t^i$ for any $a\in \mathbb{K}$.

Let $N$ be a 3--manifold and let $\phi \in H^1(N;\mathbb{Z})$ be non--trivial.
Following Turaev \cite{Tu02b} we call a ring homomorphism $\varphi:\mathbb{Z}[\pi_1(N)] \to \mathbb{K}[t^{\pm 1}]$ \emph{$\phi$--compatible}
if for any $g\in \pi_1(N)$ we have $\varphi(g)=kt^{\phi(g)}$ for some $k\in \mathbb{K}$.
Given a $\phi$--compatible  homomorphism $\varphi:\mathbb{Z}[\pi_1(N)]\to \mathbb{K}[t^{\pm 1}]$ we  consider
the  $\mathbb{K}[t^{\pm 1}]$--module
\[ H_i(N;\mathbb{K}[t^{\pm 1}])= H_i(C_*(\tilde{N})\otimes_{\mathbb{Z}[\pi_1(N)]} \mathbb{K}[t^{\pm 1}])\]
 where $\tilde{N}$ is the universal cover of $N$.
Since $\mathbb{K}[t^{\pm 1}]$ is a principal ideal domain
(PID) (cf. \cite[Proposition~4.5]{Co04}) we can decompose
\[ H_i(N;\mathbb{K}[t^{\pm 1}])\cong \bigoplus_{k=1}^l
\mathbb{K}[t^{\pm 1}]/(p_k(t))\] for $p_k(t)\in \mathbb{K}[t^{\pm 1}]$, $1 \le k \le l$. We define
$\Delta_{N,\phi,i}^\varphi\mathrel{\mathop:}=\prod_{k=1}^lp_k(t) \in \mathbb{K}[t^{\pm 1}]$.
As for twisted Alexander polynomials we write $\Delta_{N,\phi}^\varphi=\Delta_{N,\phi,1}^\varphi$.
Non--commutative Alexander polynomials have in general a high indeterminacy,
we refer to \cite[p.~367]{Co04} and
\cite[Theorem~3.1]{Fr07} for a discussion of the indeterminacy.
Note though that the degree of a non--commutative Alexander polynomial is well--defined.


The following theorem was proved for knots by Cochran \cite{Co04} and extended to 3--manifolds by
Harvey \cite{Ha05} and Turaev \cite{Tu02b}.

\begin{theorem}\label{thm:lowerboundho}
Let $N$ be a 3--manifold with empty or toroidal boundary and let $\phi\in H^1(N;\mathbb{Z})$ non--trivial.
Let $\varphi:\mathbb{Z}[\pi_1(N)]\to \mathbb{K}[t^{\pm 1}]$ be a $\phi$--compatible  homomorphism.
\begin{enumerate}
\item If the image of $\pi_1(N)\to \mathbb{K}[t^{\pm 1}]$ is non--cyclic, then
\[ \deg(\Delta_{N,\phi,0}^\varphi)=0.\]
\item If the image of $\pi_1(N)\to \mathbb{K}[t^{\pm 1}]$ is non--cyclic and if $\Delta_{N,\phi}^\varphi\ne 0$, then
\[ \deg(\Delta_{N,\phi,2}^\varphi)=0.\]
\item If $\Delta_{N,\phi}^\varphi\ne 0$, then we have the following inequality
\[ ||\phi||_T \geq \deg(\Delta_{N,\phi}^\varphi)- \deg(\Delta_{N,\phi,0}^\varphi)-\deg(\Delta_{N,\phi,2}^\varphi)\]
and equality holds if $\phi$ is a fibered class and $N\ne S^1\times D^2, N\ne S^1\times S^2$.
\end{enumerate}
\end{theorem}

We refer to \cite{Fr07} for the definition of a twisted non--commutative Alexander polynomial
and to a corresponding generalization of Theorem \ref{thm:lowerboundho}, we also refer to \cite{Fr07} for
a reinterpretation of the third statement of
Theorem \ref{thm:lowerboundho} in terms of a certain non--commutative Reidemeister torsion.

\subsection{Higher order Alexander polynomials}

We now recall the construction of what are arguably the most interesting examples of $\phi$--compatible homomorphisms from $\pi_1(N)$ to a non--commutative Laurent polynomial ring.
The ideas of this section are due to Cochran and Harvey.

\begin{theorem}\label{thm:tfa}
Let $\gamma$ be a torsion--free solvable group and let $\mathbb{F}$ be a commutative field. Then the
following hold.
\begin{enumerate}
\item  $\mathbb{F}[\Gamma]$  is an Ore domain, in particular it embeds in its classical right ring of
quotients $\mathbb{K}(\Gamma)$.
\item $\mathbb{K}(\Gamma)$ is flat over $\mathbb{F}[\Gamma]$.
\end{enumerate}
\end{theorem}

Indeed, it follows from \cite{KLM88}  that $\mathbb{F}[\Gamma]$ has no zero divisors. The first part now follows from
\cite[Corollary~6.3]{DLMSY03}. The second part is a well--known property of Ore
localizations. We call $ \mathbb{K}(\Gamma)$ the \emph{Ore localization} of
$\mathbb{F}[\Gamma]$. In \cite{COT03} the notion of a poly--torsion--free--abelian (PTFA) group is introduced,
it is well--known that these groups are torsion--free and solvable.

\begin{remark}
\begin{enumerate}
\item
It follows from Higman's theorem \cite{Hi40} that the above theorem also holds for groups which are locally indicable and amenable.
We will not make use of this, but note that throughout this section `torsion--free solvable' could be replaced by `locally indicable and amenable'.
\item  Note that the poly--torsion--free--abelian (PTFA) groups introduced in \cite{COT03} are solvable and torsion--free.
\end{enumerate}
\end{remark}

We need the following definition.

\begin{definition} Let $\pi$ be a group, $\phi:\pi \to \mathbb{Z}$ an epimorphism and
$\varphi:\pi\to \gamma$ an epimorphism to a torsion--free solvable group $\gamma$ such that there
exists a map $\phi_\Gamma:\Gamma\to \mathbb{Z}$ (which is necessarily unique) such that
 \[   \xymatrix {
 \pi \ar[dr]_{\phi} \ar[r]^{\varphi} &\Gamma \ar[d]^{\phi_\Gamma}\\& \mathbb{Z} }
     \]
 commutes. Following \cite[Definition~1.4]{Ha06} we call $(\varphi,\phi)$
an {\em admissible pair}.
\end{definition}

Now let $(\varphi:\pi_1(N)\to \Gamma,\phi)$ be an admissible pair for $\pi_1(N)$. In the following we
denote $\ker\{\phi:\Gamma\to \mathbb{Z}\}$ by $\Gamma'(\phi)$. When the homomorphism $\phi$ is understood we will
write $\Gamma'$ for $\Gamma'(\phi)$. Clearly $\Gamma'$ is still solvable and torsion--free. Let
$\mathbb{F}$ be any commutative field and $\mathbb{K}(\Gamma')$ the Ore localization of $\mathbb{F}[\Gamma']$. Pick an element $\mu
\in \gamma$ such that $\phi(\mu)=1$. Let $\gamma:\mathbb{K}(\Gamma')\to \mathbb{K}(\Gamma')$ be the homomorphism given by
$\gamma(a)=\mu a\mu^{-1}$. Then we get a ring homomorphism
\[ \begin{array}{rcl} \mathbb{Z}[\Gamma]&\to& \mathbb{K}(\Gamma')_\gamma[t^{\pm 1}]\\[1mm]
    g &\mapsto& (g\mu^{-\phi(g)}t^{\phi(g)}), \mbox{ for $g\in \gamma$.}\end{array} \]
    We denote this ring homomorphism again by $\varphi$.
    It is clear that $\varphi$ is $\phi$--compatible. Note that the ring $\mathbb{K}(\Gamma')[t^{\pm 1}]$ and hence the
above representation depends on the choice of $\mu$. We will nonetheless suppress $\mu$ in the
notation since different choices of splittings give isomorphic rings.
We will refer to a non--commutative Alexander polynomial corresponding to such a group homomorphism
as a \emph{higher order Alexander polynomial}.

An important example of admissible pairs is provided by Harvey's rational derived series of a group
$\gamma$ (cf. \cite[Section~3]{Ha05}). Let $\gamma_r^{(0)}=\gamma$ and define inductively
\[ \gamma_r^{(n)}=\big\{ g\in \gamma_r^{(n-1)} | \, g^k \in \big[\gamma_r^{(n-1)},\gamma_r^{(n-1)}\big] \mbox{ for some }k\in \mathbb{Z} \setminus \{0\} \big\}.\]
Note that $\gamma_r^{(n-1)}/\gamma_r^{(n)}\cong
\big(\gamma_r^{(n-1)}/\big[\gamma_r^{(n-1)},\gamma_r^{(n-1)}\big]\big)/\mbox{$\mathbb{Z}$--torsion}$.
 By \cite[Corollary~3.6]{Ha05} the quotients $\Gamma/\gamma_r^{(n)}$ are solvable and torsion--free
for any $\gamma$ and any $n$. If $\phi:\Gamma\to \mathbb{Z}$ is an epimorphism, then  $(\Gamma\to \Gamma/\gamma_r^{(n)},\phi)$ is
an admissible pair for $(\Gamma,\phi)$ for any $n>0$.

For example if $K$ is a knot, $\gamma=\pi_1(S^3\setminus \nu K)$, then it follows from \cite{St74} that
$\Gamma^{(n)}_r=\Gamma^{(n)}$, i.e. the rational derived series equals the ordinary derived series (cf. also
\cite{Co04} and \cite{Ha05}).

\begin{remark}
The Achilles heel of the higher order Alexander polynomials is that they are unfortunately difficult to compute
in practice. We refer to \cite{Sa07} for some ideas on how to compute higher order Alexander polynomials in some cases.
\end{remark}


\subsection{Comparing different $\phi$--compatible maps}\label{section:comparison}

\noindent We now recall a definition from \cite{Ha06}.

\begin{definition}
Let $N$ be a 3--manifold with empty or toroidal boundary.
We write $\pi=\pi_1(N)$. Let $\phi:\pi\to \mathbb{Z}$ an epimorphism. Furthermore let $\varphi_1:\pi \to
\gamma_1$ and $\varphi_2:\pi \to \gamma_2$ be epimorphisms to torsion--free solvable groups $\gamma_1$
and $\gamma_2$. We call $(\varphi_1,\varphi_2,\phi)$ an {\em admissible triple} for $\pi$ if there
exist epimorphisms $\varphi_{2}^1:\gamma_1\to \gamma_2$  and $\phi_2:\gamma_2\to
\mathbb{Z}$ such that $\varphi_2=\varphi_{2}^1\circ \varphi_1$, and $\phi=\phi_2\circ \varphi_2$.
\end{definition}

\noindent The situation can be summarized in the following diagram
 
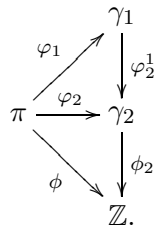
\[   \xymatrix { &\gamma_1\ar[d]^{\varphi_2^1}\\ \pi \ar[dr]_{\phi} \ar[ur]^{\varphi_1} \ar[r]^{\varphi_2}
&\gamma_2 \ar[d]^{\phi_2}\\& \mathbb{Z}. }
     \]
Note that in particular $(\varphi_i,\phi), i=1,2$ are admissible pairs for $\pi$.
The following theorem is perhaps the most striking feature of higher order Alexander polynomials.
In light of Theorem \ref{thm:lowerboundho} the statement can be summarized as saying that higher order Alexander polynomials corresponding to larger groups
give better bounds on the Thurston norm.

 \begin{theorem}
Let $N$ be a 3--manifold whose boundary is a (possibly empty) collection of tori. Let
$(\varphi_1,\varphi_2,\phi)$ be an admissible triple for $\pi_1(N)$.
Suppose that $\Delta_{N,\phi}^{\varphi_2}\ne 0$, then it follows that $\Delta_{N,\phi}^{\varphi_1}\ne 0$.
We write
\[ d_i\mathrel{\mathop:}=\deg(\Delta_{N,\phi}^{\varphi_i})- \deg(\Delta_{N,\phi,0}^{\varphi_i})-\deg(\Delta_{N,\phi,2}^{\varphi_i}), i=1,2 \]
Then the following holds:
\[ d_1 \geq d_2.\]
Furthermore, if the ordinary Alexander polynomial $\Delta_N^\phi\in \mathbb{Z}[t^{\pm 1}]$ is non--trivial, then
 $d_1-d_2$ is an even integer.
\end{theorem}

\begin{proof}
The fact that $\Delta_{N,\phi}^{\varphi_2}\ne 0$ implies that $\Delta_{N,\phi}^{\varphi_1}\ne 0$
and the  inequality $d_2\geq d_1$  were first proved for knots by Cochran \cite{Co04}.
 Cochran's result were  then extended to the case of 3--manifolds by Harvey \cite{Ha06} (cf. also \cite{Fr07}).
Finally the fact that $d_2-d_1$ is an even integer when $\Delta_N^\phi\ne 0$ is proved in \cite{FK08a}.
\end{proof}

The strong relationship between the Thurston norm and higher order Alexander polynomials
is also confirmed by the following result (cf. \cite{FH07}).

\begin{theorem}
Let $N$ be a 3--manifold with empty or toroidal boundary, let $\varphi:\pi_1(N) \to \gamma$ be an epimorphism to a torsion--free solvable group such that the abelianization $\pi_1(N)\to F\mathrel{\mathop:}=H_1(N;\mathbb{Z})/\mbox{torsion}$ factors through $\varphi$.
Then  the map
\[ \begin{array}{rcl} H^1(N;\mathbb{Z})=\mbox{Hom}(F,\mathbb{Z})&\to & \mathbb{Z}_{\geq 0} \\
\phi &\mapsto &\max\{0,\deg(\Delta_{N,\phi}^{\varphi})- \deg(\Delta_{N,\phi,0}^{\varphi})-\deg(\Delta_{N,\phi,2}^{\varphi})\}\end{array} \]
defines a seminorm on $H^1(N;\mathbb{Z})$ which gives a lower bound on the Thurston norm.
\end{theorem}
\subsection{Miscellaneous applications of higher order Alexander polynomials}

In this section we quickly summarize various applications of higher order Alexander polynomials and related invariants to various aspects of low--dimensional topology:
\begin{enumerate}
\item Leidy \cite{Lei06} studied the relationship between higher order Alexander modules and non--commutative Blanchfield pairings.
\item Leidy and Maxim (\cite{LM06} and \cite{LM08}) studied  higher order Alexander polynomials of plane curve complements.
\item Cochran and Taehee Kim \cite{CT08} showed that given a knot with genus greater than one, the higher order
Alexander polynomials do not determine the concordance class of a knot.
\item Sakasai \cite{Sa06,Sa08} and Goda--Sakasai \cite{GS08} studied applications of higher order Alexander invariants to homology cylinders and sutured manifolds. For example higher order Alexander invariants can be used to give obstructions to homology cylinders being products.
\end{enumerate}

\section{Open questions and problems}\label{section:question}

We conclude this survey paper with a list of open questions and problems.

\begin{enumerate}
\item Using elementary ideals one can define the twisted $k$--th Alexander polynomial, generalizing the $k$--th Alexander polynomial of a knot $K\subset S^3$.
What information do these invariants contain?
\item  Let $K\subset S^3$ a knot and let $\gamma:\pi_1(S^3\setminus \nu K)\to \mbox{SL}(k,R)$ be a representation, where $R$ is a Noetherian UFD with possibly trivial involution.
Does it follow that $\Delta_K^\gamma$ is reciprocal, i.e. does it hold that $\Delta_K^\gamma \doteq \overline{\Delta_K^\gamma}$?
Note that this holds for unitary representations (cf. Section \ref{section:duality}, \cite{Ki96}, \cite{KL99a}) and for all calculations known to the authors. \\
Added in proof: This question was answered in the negative by Hillman, Silver and Williams \cite{HSW09}, cf. also the remark after Proposition \ref{prop:dualitytau}.
\item Can any two knots or links be distinguished using twisted Alexander polynomials?
\item If $(N,\phi)$ is non--fibered, does there exist a representation $\gamma:\pi_1(N)\to \mbox{GL}(k,R)$ such that
$\Delta_{N}^{\gamma \otimes \phi}$ is not monic?
\item If $(N,\phi)$ is non--fibered, does there exist a representation $\gamma:\pi_1(N)\to \mbox{SL}(2,\mathbb{C})$ such that
$\tau(N,{\gamma\otimes \phi})$ is not monic? (cf. e.g. \cite[Problem~1.1]{GM03}).
\item If $(N,\phi)$ is non--fibered, does there exist a representation $\gamma:\pi_1(N)\to \mbox{GL}(k,R)$ such that
$\Delta_{N}^{\gamma \otimes \phi}$ is zero?
\item Let $K\subset S^3$ be any knot, does the twisted Reidemeister torsion of \cite{GKM05} corresponding to a generic faithful representation detect fiberedness? (cf. \cite[p.~452]{Mo08} for some calculations).
\item Let $N$ be a 3--manifold with empty or toroidal boundary, $N\ne S^1\times D^2, S^1\times S^2$, let $\phi \in H^1(N;\mathbb{Z})$ and let $\gamma:\pi_1(N)\to \mbox{GL}(k,R)$ be a  representation such that $\Delta_{N}^{{\gamma\otimes \phi}}\ne 0$. Does it follow that
\[\deg(\tau(N,{\gamma\otimes \phi}))= \deg(\Delta_{N,1}^{\gamma\otimes \phi})-\deg(\Delta_{N,0}^{\gamma\otimes \phi})-\deg(\Delta_{N,2}^{\gamma\otimes \phi})  \]
has the parity of $k||\phi||_T$? Note that this holds for fibered $(N,\phi)$ and for the untwisted Alexander polynomials of a knot.\\
Added in proof: this also holds for hyperbolic knots and the canonical $SL(2,\C)$ representation. 
\item Does the twisted Alexander polynomial detect the Thurston norm of a given $\phi \in H^1(N;\mathbb{Z})$?
\item Let $K$ be a hyperbolic knot and $\rho:\pi_1(S^3\setminus \nu K)\to \mbox{SL}(2,\mathbb{C})$ the unique discrete faithful representation.
\begin{enumerate}
\item  Is $\Delta_K^\rho$ non--trivial?
\item Does $\deg(\Delta_K^\rho)$ determine the genus of $K$?
\item Is $K$ fibered if $\tau(K,\rho)$ is monic?
\end{enumerate}
Note that the unique discrete representation is over a number field which for many knots can be obtained explicitly with Snappea.
These questions can therefore be answered for small crossing knots.\\
Added in proof: the answer to all three questions is yes, if $K$ has at most 13 crossings (\cite{DFJ10}).
\item Does there exist a knot $K\subset S^3$ and a nonabelian representation  $\gamma$ such that $\Delta_K^\gamma$ is trivial?
\item Are there knots for which Kitayama's lower bounds on the free genus of a knot (cf. \cite{Kiy08a})
are larger than the bound on the ordinary genus obtained in \cite{FK06}?
\item Find a practical algorithm for computing higher order Alexander polynomials.
\item Do higher order Alexander polynomials detect mutation?
\item Does there exist a twisted version of Turaev's torsion function?
\item Use twisted Alexander polynomials to determine which knots with up to twelve crossings are doubly slice.
\item Can the results of \cite{HK79} and \cite{Hat80} regarding Alexander polynomials of amphichiral knots be generalized to twisted Alexander polynomials?
\end{enumerate}


\begin{thebibliography}{DLMSY03}
\bibitem[Ag08]{Ag08}
I. Agol, {\em Criteria for virtual fibering}, Journal of Topology 1: 269-284 (2008)
\bibitem[Al28]{Al28}
J.W. Alexander, {\em Topological invariants of knots and links}, Trans. Amer. Math. Soc. 30 (1928), no. 2, 275--306
\bibitem[Bl57]{Bl57}
R. Blanchfield, {\em Intersection theory of manifolds with operators with applications
to knot theory}, Ann. of Math. (2) 65: 340--356 (1957)
\bibitem[BF08]{BF08}
H. Boden and S. Friedl, {\em Metabelian $SL(n,\mathbb{C})$ representations of knot groups},
Pac. J. Math, Vol. 238, 7--25 (2008)
\bibitem[Bu67]{Bu67}
G. Burde,  {\em Darstellungen von Knotengruppen},  Math. Ann.  173  (1967) 24--33

\bibitem[CG86]{CG86}
A. Casson and C. Gordon, {\em Cobordism of classical knots}, Progr. Math., 62, A la recherche de la
topologie perdue, 181--199, Birkh\"auser Boston, Boston, MA (1986)
\bibitem[Ch03]{Ch03} J. Cha, {\em Fibred knots and twisted Alexander invariants},
Transactions of the AMS 355: 4187--4200 (2003)
\bibitem[Ch09]{Ch09}
J. C. Cha, {\em Link concordance, homology cobordism, and Hirzebruch-type defects from iterated p-covers}, to appear in the Journal of the European Mathematical Society (2009)
\bibitem[CF10]{CF10}
J. Cha and S. Friedl, {\em
Twisted Reidemeister torsion and link concordance}, Preprint (2010)
\bibitem[Chp74]{Chp74}
T. A. Chapman, {\em Topological invariance of Whitehead torsion},
Amer. J. Math. Vol 96, No. 3: 488--497 (1974)

\bibitem[COT03]{COT03}
T. Cochran, K. Orr and  P. Teichner, {\em Knot concordance, Whitney towers and $L\sp 2$-signatures},
Ann. of Math. (2)  157,  no. 2: 433--519 (2003)
\bibitem[CT08]{CT08}
T. Cochran and T. Kim, {\em  Higher-order Alexander invariants and filtrations of the knot concordance group},
Trans. Amer. Math. Soc., 360 no. 3: 1407--1441 (2008)
\bibitem[Co04]{Co04}
T. Cochran,  {\em Noncommutative knot theory}, Algebr. Geom.
Topol. \textbf{4} (2004), 347--398.
\bibitem[CFT09]{CFT09}
T. Cochran, S. Friedl, P. Teichner, {\em
New constructions of slice links},
Comment. Math. Helv. Volume 84, Issue 3 (2009), 617-638
\bibitem[CF07]{CF07}
A. Cogolludo and  V. Florens, {\em  Twisted Alexander polynomials of plane algebraic curves},  J. Lond. Math. Soc. (2)  76  (2007),  no. 1, 105--121
\bibitem[CS08]{CS08}
D. Cohen and  A. Suciu, {\em
The boundary manifold of a complex line arrangement}, Geometry \& Topology Monographs 13 (2008) 105-146.
\bibitem[CF77]{CF77}
R. Crowell and  R. Fox, {\em Introduction to knot theory}, Graduate Texts in Mathematics, No. 57.
Springer-Verlag, New York-Heidelberg, 1977.
\bibitem[dRh68]{dRh68}
G. de Rham, {\em Introduction aux polyn\^omes d'un n{\oe}ud}, Enseignement Math. (2)  13, 187--194 (1968).

\bibitem[DLMSY03]{DLMSY03}
J. Dodziuk, P. Linnell, V. Mathai, T. Schick and S. Yates, {\em Approximating $L^2$-invariants,
and the Atiyah conjecture}, Preprint Series SFB 478 M\"unster, Germany. Communications on Pure
and Applied Mathematics, vol. 56, no. 7: 839-873 (2003)
\bibitem[Do96]{Do96}
S. Donaldson, {\em Symplectic submanifolds and almost-complex geometry}, J. Diff.
Geom., 44 (1996), 666-705
\bibitem[Do99]{Do99}
S. Donaldson, {\em Topological field theories and formul� of Casson Meng-Taubes}, Proceedings
of the Kirbyfest, Geom. Topol. Monogr. 2, 1999, 87�102
\bibitem[DHY09]{DHY09}
J. Dubois, V. Huynh and Y. Yamaguchi, {\em
Non-abelian Reidemeister torsion for twist knots},
 J. Knot Theory Ramifications 18 (2009), no. 3, 303--341
\bibitem[DY09]{DY09}
J. Dubois  and Y. Yamaguchi, {\em
 Multivariable Twisted Alexander Polynomial for hyperbolic three-manifolds with boundary}, Preprint (2010)
\bibitem[Du01]{Du01} N. Dunfield, {\em Alexander and Thurston norm of fibered 3--manifolds}, Pacific J. Math.,
200, no. 1: 43--58 (2001)
\bibitem[DFJ10]{DFJ10}
N. Dunfield, S. Friedl and N. Jackson,
{\em Twisted Alexander polynomials of hyperbolic knots}, in preparation (2010)
\bibitem[Ei07]{Ei07}
M. Eisermann, {\em Knot colouring polynomials},
Pacific Journal of Mathematics 231 (2007), 305-336.
\bibitem[El08]{El08}
R. Elliot, {\em  Alexander Polynomials of Periodic Knots: A Homological Proof and Twisted Extension}, Princeton Undergraduate Thesis (2008)
\bibitem[Fo56]{Fo56}
R. H. Fox, {\em Free differential calculus III}, Annals of Math. 64 (1956), 407--419
\bibitem[Fo62]{Fo62}
R. H. Fox, {\em  A quick trip through knot theory}, Topology of Three Manifolds and Related
Topics (Prentice Hall 1962), 120�167
\bibitem[Fo70]{Fo70}
R. H. Fox, {\em Metacyclic invariants of knots and links}, Canad. J. Math. 22(1970) 193�207
\bibitem[FM66]{FM66}
R. H. Fox and J. W. Milnor, {\em Singularities of 2--spheres in 4--space and cobordism of knots},
Osaka J. Math 3: 257--267 (1966)
\bibitem[Fr03]{Fr03}
S. Friedl, {\em Eta invariants as sliceness obstructions and their relation to Casson
Gordon invariants}, Thesis, Brandeis University (2003)
\bibitem[Fr04]{Fr04}
S. Friedl, {\em Eta invariants as sliceness obstructions and their relation to Casson-Gordon invariants},
Algebraic and Geometric Topology, Vol. 4: 893-934 (2004).
\bibitem[Fr07]{Fr07}
S. Friedl, {\em Reidemeister torsion, the Thurston norm and Harvey's invariants},
 Pacific Journal
of Mathematics, Vol. 230: 271-296 (2007)
\bibitem[Fr09]{Fr09}
S. Friedl, {\em KnotTwister}, \texttt{http://www.warwick.ac.uk/\~\,masgaw/index.html} (2009)
\bibitem[FH07]{FH07}
S. Friedl and S. Harvey, {\em
Non-commutative Multivariable Reidemeister Torsion and the Thurston Norm},
Alg. Geom. Top., Vol. 7: 755-777 (2007)
\bibitem[FK06]{FK06}
 S. Friedl and  T. Kim, \emph{Thurston norm, fibered manifolds and twisted Alexander
polynomials}, Topology, Vol. 45: 929-953 (2006)
\bibitem[FK08a]{FK08a}
S. Friedl and T. Kim, {\em  Twisted Alexander norms give lower bounds on the Thurston norm},
Trans. Amer. Math. Soc. 360 (2008), 4597-4618
\bibitem[FK08b]{FK08b}
S. Friedl and T. Kim, {\em The parity of the Cochran-Harvey invariants of
3-manifolds},
    Trans. Amer. Math. Soc. 360 (2008), 2909-2922.
\bibitem[FLM09]{FLM09}
S. Friedl, C. Leidy and L. Maxim, {\em $L^2$--Betti numbers of plane algebraic curves},
Michigan Math. Journal, 58 (2009), no. 2, 291-301
\bibitem[FT05]{FT05}
S. Friedl and P. Teichner, {\em
    New topologically slice knots},
Geometry and Topology, Volume 9 (2005) Paper no. 48, pages 2129--2158
\bibitem[FV06]{FV06}
S. Friedl and  S. Vidussi,
{\em  Symplectic $S^1 \times N^3$ and subgroup separability},
Oberwolfach Reports, Volume 3, Issue 3 (2006), 2166 - 2168
\bibitem[FV07a]{FV07a}
S. Friedl and  S. Vidussi,
{\em Nontrivial Alexander polynomials of knots and links},  Bull. London Math. Soc., Vol. 39: 614--6222 (2007)
\bibitem[FV07b]{FV07b}
S. Friedl and  S. Vidussi, {\em Symplectic 4-manifolds with a free circle action}, Preprint (2007)
\bibitem[FV08a]{FV08a}
S. Friedl and  S. Vidussi, {\em Twisted Alexander polynomials and
symplectic structures}, Amer. J. Math. 130, no 2: 455-- 484 (2008)
\bibitem[FV08b]{FV08b}
S. Friedl and  S. Vidussi, {\em Symplectic $S^{1} \times N^3$,
surface subgroup separability, and vanishing Thurston
norm},
J. Amer. Math. Soc. 21 (2008), 597-610.
 \bibitem[FV08c]{FV08c}
S. Friedl and  S. Vidussi,
 {\em Twisted Alexander polynomials detect fibered 3--manifolds}, Preprint (2008)
\bibitem[FV09]{FV09}
S. Friedl and  S. Vidussi,
 {\em Twisted Alexander polynomials, symplectic 4--manifolds and surfaces of minimal complexity}, Banach Center Publ. 85 (2009), 43-57
 \bibitem[FV10]{FV10}
S. Friedl and  S. Vidussi,
 {\em Twisted Alexander polynomials and fibered 3-manifolds}, to be published by the Proceedings of the Georgia International Topology Conference (2010)
\bibitem[GL83]{GL83}
P. Gilmer and C.  Livingston,
{\em On embedding $3$-manifolds in $4$-space},
Topology 22 (1983), no. 3, 241--252.
\bibitem[GKM05]{GKM05} H. Goda, T. Kitano and  T. Morifuji, {\em Reidemeister Torsion, Twisted Alexander Polynomial and Fibred Knots}, Comment. Math. Helv.  80,  no. 1: 51--61 (2005)
\bibitem[GM03]{GM03} H. Goda and  T. Morifuji, {\em Twisted Alexander polynomial for ${\rm SL}(2,\mathbb{C})$-representations and fibered knots},
C. R. Math. Acad. Sci. Soc. R. Can.  25  (2003),  no. 4, 97--101
\bibitem[GP05]{GP05}
H. Goda and  A. Pajitnov, {\em     Twisted Novikov homology and circle-valued Morse theory for knots and links},
 Osaka Journal of Mathematics, vol. 42 No. 3, 2005
\bibitem[GS08]{GS08}
H. Goda, T. Sakasai, {\em  Homology cylinders in knot theory}, Preprint (2008)
\bibitem[GS91]{GS91}
F. Gonz\'alez-Acu\~na and H. Short, {\em Cyclic branched coverings of knots and homology
spheres}, Revista Math. 4 (1991), 97 - 120.
\bibitem[Go72]{Go72}
C. McA. Gordon,
{\em Knots whose branched cyclic coverings have periodic homology}, Trans. Amer. Math. Soc. 168 (1972), 357--370.
\bibitem[Go78]{Go78}
 C. McA. Gordon, {\em Some aspects of classical knot theory}, Knot theory (Proc.
Sem., Plans-sur-Bex, 1977), Lecture Notes in Math. 685:1--65 (1978)
\bibitem[GL89]{GL89}
C. McA. Gordon and J. Luecke, {\em Knots are determined by their
complements}, J. Amer. Math. Soc. \textbf{2} (1989), no. 2, 371--415.

\bibitem[HK79]{HK79}
R. Hartley and A. Kawauchi, {\em Polynomials of amphicheiral knots}, Math. Ann. 243 (1979), no. 1, 63--70.

\bibitem[Hat79]{Hat79}
R. Hartley, {\em Metabelian representations of knot groups}, Pacific J. Math. 82(1979) 93�
104.
\bibitem[Hat80]{Hat80}
R. Hartley, {\em Invertible amphicheiral knots}, Math. Ann.  252  (1979/80), no. 2, 103--109.

\bibitem[Hat81]{Hat81}
R. Hartley, {\em Knots with free period}, Canad. J. Math. 33 (1981), no. 1, 91--102.
\bibitem[Ha05]{Ha05}
S. Harvey, {\em Higher--order polynomial invariants of 3--manifolds giving lower bounds for the
Thurston norm}, Topology \textbf{44} (2005), 895--945.
\bibitem[Ha06]{Ha06}
S. Harvey, {\em Monotonicity of degrees of generalized Alexander polynomials of groups and
3--manifolds},
Math. Proc. Camb. Phil. Soc., Volume 140, Issue 03, (2006) 431--450.
\bibitem[Ha08]{Ha08}
S. Harvey, {\em
Homology Cobordism Invariants and the Cochran-Orr-Teichner Filtration of the Link Concordance Group},
Geom. Topol., Vol 12 (2008), 387--430.
\bibitem[HKL08]{HKL08}
C. Herald, P. Kirk and  C. Livingston, {\em
Metabelian representations, twisted Alexander polynomials, knot slicing, and mutation}, Preprint (2008),
to appear in Math. Z.
\bibitem[HP05]{HP05}
M. Heusener and J. Porti, {\em
Deformations of reducible representations of 3-manifold groups into $PSL_2(\mathbb{C})$},  Algebr. Geom. Topol. 5 (2005) 965-997
\bibitem[HLN06]{HLN06} J. A. Hillman, C. Livingston and S. Naik,
{\em Twisted Alexander polynomials of periodic knots},   Algebr. Geom. Topol.  6  (2006), 145--169
\bibitem[Hi40]{Hi40}
G. Higman, {\em The units of group-rings}, Proc. London Math. Soc. (2) 46, (1940) 231--248.
\bibitem[Hi02]{Hi02}
J. Hillman, {\em Algebraic invariants of links}, Series on Knots and Everything, 32. World
Scientific Publishing Co., Inc., River Edge, NJ, 2002.
\bibitem[HSW09]{HSW09}
J. Hillman, D. Silver and S. Williams, {\em On Reciprocality of Twisted Alexander Invariants}, Preprint (2009)
\bibitem[HM08]{HM08}
M. Hirasawa and K. Murasugi, {\em
Evaluations of the twisted Alexander polynomials of 2-bridge knots at $\pm 1$}, Preprint (2008)
\bibitem[HM09a]{HM09a}
M. Hirasawa and K. Murasugi, {\em
Twisted Alexander polynomials of 2-bridge knots associated to metacyclic representations}, Preprint (2009)
\bibitem[HM09b]{HM09b}
M. Hirasawa and K. Murasugi, {\em
Twisted Alexander polynomials of 2-bridge knots associated to metabelian representations}, Preprint (2009)
\bibitem[HKMS09]{HKMS09}  K. Horie, T. Kitano, M. Matsumoto and M. Suzuki,
{\em A partial order on the set of prime knots with up to eleven crossings}, Preprint (2009)
\bibitem[HL07]{HL07}
V. Q. Huynh and  T. Le, {\em
On the twisted Alexander polynomial and the A-polynomial of 2-bridge knots}, Preprint (2007)
\bibitem[HL08]{HL08}
V. Q. Huynh and  T. Le, {\em Twisted Alexander polynomial of links in the projective space},
J. Knot Theory Ramifications  17  (2008),  no. 4, 411--438
\bibitem[In00]{In00}
G. Indurskis, {\em Das getwistete Alexanderpolynom, Verallgemeinerung
einer klassischen Knoteninvariante}, Diploma thesis (2000)
\bibitem[Je08]{Je08}
H. Jebali,
{\em Module d'Alexander et repr\'esentations m\'etab\'eliennes},
Annales de la facult\'e des sciences de Toulouse S\'er. 6, 17 no. 4 (2008), p. 751-764.
\bibitem[JW93]{JW93}
B. Jiang and S. Wang, {\em Twisted topological invariants
associated with representations}, Topics in knot theory (Erzurum,
1992), 211--227, NATO Adv. Sci. Inst. Ser. C Math. Phys. Sci.,
399, Kluwer Acad. Publ., Dordrecht, 1993.
\bibitem[Ka06]{Ka06}
T. Kadokami, {\em Reidemeister torsion and lens surgeries on knots in homology 3-spheres. I}, Osaka J. Math. 43
(2006) 823--837.
\bibitem[Ka07]{Ka07}
T. Kadokami, {\em Reidemeister torsion of Seifert fibered homology lens spaces and Dehn surgery}, Alg. Geom.
Topol. 7 (2007) 1509-1529.
\bibitem[Ka77]{Ka77}
A. Kawauchi, {\em On quadratic forms of 3-manifolds}, Invent. Math. 43 (1977), 177--198
\bibitem[Ka78]{Ka78}
A. Kawauchi, {\em On the Alexander polynomials of cobordant links}, Osaka J. Math. 15 (1978),
no. 1, 151--159.
\bibitem[Ka96]{Ka96}
A. Kawauchi, {\em Survey on knot theory}, Birkh\"auser (1996)
\bibitem[Ke75]{Ke75}
C. Kearton, {\em Cobordism of knots and Blanchfield duality}, J. London Math.
Soc. (2) 10, no. 4: 406--408 (1975)
\bibitem[KL99a]{KL99a} P. Kirk and  C. Livingston, {\em Twisted Alexander invariants, Reidemeister
torsion and Casson--Gordon invariants}, Topology 38, no. 3: 635--661 (1999)
\bibitem[KL99b]{KL99b}
P. Kirk and C. Livingston, {\em Twisted knot polynomials:
inversion, mutation and concordan}, Topology \textbf{38} (1999),
no. 3, 663--671.
\bibitem[Ki96]{Ki96}
T. Kitano, {\em Twisted Alexander polynomials and Reidemeister
torsion}, Pacific J. Math. \textbf{174} (1996), no. 2, 431--442.
\bibitem[KM05]{KM05}
T. Kitano and  T.  Morifuji, {\em Divisibility of twisted Alexander polynomials and fibered knots},  Ann. Sc. Norm. Super. Pisa Cl. Sci. (5)  4  (2005),  no. 1, 179--186.
\bibitem[KS05a]{KS05a}
T. Kitano and M. Suzuki, {\em A partial order in the knot table}, Experiment. Math. 14 (2005)
385--390
\bibitem[KS05b]{KS05b}
T. Kitano and  M. Suzuki, {\em Twisted Alexander polynomials and a partial order on the set of prime knots},
Geometry and Topology Monographs 13 (2008), Groups, Homotopy and Configuration Spaces (Tookyo 2005), 307--322
\bibitem[KS08]{KS08}
T. Kitano and  M. Suzuki, {\em A partial order in the knot table II}, Acta Mathematica Sinica 24 (2008), 1801--1816
\bibitem[KSW05]{KSW05}
T. Kitano, M. Suzuki and M. Wada, {\em Twisted Alexander polynomial and surjectivity of a group
homomorphism},  Algebr. Geom. Topol. 5 (2005), 1315--1324.
\bibitem[Kiy08a]{Kiy08a}
T. Kitayama, {\em Normalization of twisted Alexander invariants}, Preprint (2008)
\bibitem[Kiy08b]{Kiy08b}
T. Kitayama, {\em  Symmetry of Reidemeister torsion on $SU_2$-representation spaces of knots}, Preprint (2008)
\bibitem[Kiy09]{Kiy09}
T. Kitayama, {\em Reidemeister torsion for linear representations and Seifert fibered surgery on knots},
 Topology Appl. 156 (2009), no. 15, 2496--2503.
\bibitem[Kr99]{Kr99} P. Kronheimer, {\em Minimal genus in $S\sp 1\times
M\sp 3$},  Invent. Math.  135,  no. 1: 45--61 (1999)
\bibitem[KM08]{KM08}
P. Kronheimer and  T. Mrowka, {\em
Knots, sutures and excision}, Preprint (2008)
\bibitem[KLM88]{KLM88}
P. H. Kropholler, P. A. Linnell, J. A. Moody, {\em Applications of a new $K$-theoretic theorem to soluble group rings}, Proc. Amer. Math. Soc. Vol 104, No. 3: 675--684 (1988)
\bibitem[KT09]{KT09}
\c C. Kutluhan and  C. Taubes, {\em Seiberg-Witten Floer homology and symplectic forms on $S^1 \times M^3$},
Geom. Topol. 13 (2009), no. 1, 493--525.
\bibitem[Lei06]{Lei06}
C. Leidy, {\em   	
Higher-order linking forms for knots},
Commentarii Mathematici Helvetici 81 (2006), 755-781
\bibitem[LM06]{LM06}
C. Leidy and L. Maxim, {\em Higher-order Alexander invariants of plane algebraic curves},
 IMRN, Volume 2006 (2006), Article ID 12976, 23 pages.
\bibitem[LM08]{LM08}
C. Leidy and L. Maxim, {\em
Obstructions on fundamental groups of plane curve complements},
Real and Complex Singularities, Contemporary Mathematics, 459:117--130 (2008)

\bibitem[Let00]{Let00}
C. Letsche, {\em
An obstruction to slicing knots using the eta invariant},
Math. Proc. Cambridge Phil. Soc. 128, no. 2: 301--319 (2000)
\bibitem[Lev94]{Lev94}
J. Levine, {\em Links invariants via the eta invariant}, Commentarii Mathematici Helvetici
\bibitem[LX03]{LX03}
W. Li and L. Xu, {\em
 Counting $\mbox{SL}_2(\mathbb{F}_{2^s})$ representations of Torus Knot Groups},
	Acta Mathematica Sinica, Volume 19 (2003), 233--244.
\bibitem[LZ06a]{LZ06a}
W. Li and W. Zhang, {\em An $L\sp 2$-Alexander-Conway invariant for knots and the volume conjecture},  Differential geometry and physics,  303--312, Nankai Tracts Math., 10, World Sci. Publ., Hackensack, NJ, 2006
\bibitem[LZ06b]{LZ06b}
W. Li and W. Zhang, {\em An $L\sp 2$-Alexander invariant for knots}. Commun. Contemp. Math. 8 (2006), no. 2, 167--187.
\bibitem[Lib82]{Lib82}
A. Libgober, {\em Alexander polynomial of plane algebraic curves and cyclic multiple planes},
Duke Math. Journal 49 (1982), no. 4, 833--851.
\bibitem[Lin01]{Lin01} X. S. Lin, {\em Representations of knot groups and twisted
Alexander polynomials}, Acta Math. Sin. (Engl. Ser.)  17,  no. 3: 361--380 (2001)
\bibitem[Liv09]{Liv09}
C. Livingston, {\em The concordance genus of a knot, II}, Algebraic \& Geometric Topology 9 (2009) 167-185

\bibitem[LN91]{LN91}
D. Long and  G. Niblo, {\em Subgroup separability and $3$-manifold groups}, Math. Z. 207 (1991), no. 2,
209--215.
\bibitem[LS03]{LS03} A. Lubotzky and  D. Segal, {\em Subgroup growth}, Progress in Mathematics, 212. Birkh\"auser
Verlag, Basel, 2003.
\bibitem[L\"u02]{Lu02} W. L\"uck, {\em $L\sp 2$-invariants: Theory and Applications to Geometry and
$K$-Theory}, Ergebnisse der Mathematik und ihrer Grenzgebiete. 3.
Folge. A Series of Modern Surveys in Mathematics, 44.
Springer-Verlag, Berlin, 2002.
\bibitem[McM02]{McM02} C. T. McMullen, {\em The Alexander polynomial of a 3--manifold and the Thurston
norm on cohomology}, Ann. Sci. Ecole Norm. Sup. (4) 35, no. 2: 153--171 (2002)
\bibitem[MP09]{MP09}
P. Menal-Ferrer and  J. Porti, {\em Twisted cohomology for  hyperbolic three manifolds}, in preparation  (2009)
\bibitem[MT96]{MT96} G. Meng and  C. H. Taubes, {\em SW = Milnor torsion}, Math. Res. Lett. 3: 661--674 (1996)
\bibitem[Mi62]{Mi62}
J. Milnor, {\em A duality theorem for Reidemeister torsion},
Ann. of Math. (2) 76 (1962) 137--147.
\bibitem[Mi66]{Mi66}
J. Milnor, {\em Whitehead torsion}, Bull. Amer. Math. Soc. 72 (1966), 358--426.
\bibitem[Mo01]{Mo01}
T. Morifuji, {\em
Twisted Alexander polynomial for the braid group},
Bull. Austral. Math. Soc. 64 (2001), no. 1, 1--13.
\bibitem[Mo07]{Mo07}
T. Morifuji, {\em A Torres Condition for Twisted Alexander Polynomials},
Publ. Res. Inst. Math. Sci. Volume 43, Number 1 (2007), 143-153.
\bibitem[Mo08]{Mo08}
T. Morifuji, {\em Twisted Alexander polynomials of twist knots for nonabelian representations}, Bull. Sci. Math. 132: 439--453 (2008)
\bibitem[Mu67]{Mu67}
K. Murasugi, {\em On a certain numerical invariant of link types}, Trans. Amer. Math. Soc. 117 (1967), 387--422
\bibitem[Mu71]{Mu71}
K. Murasugi, {\em
On periodic knots}, Comment. Math. Helv. 46 (1971), 162-174.
\bibitem[Mu03]{Mu03}
K. Murasugi, {\em Lectures at the Workshop in Osaka City University}, December 2003
\bibitem[Mu06]{Mu06}
K. Murasugi, {\em Classical knot invariants and elementary number theory},
Cont. Math. 416, 167--197 (2006)


\bibitem[Na78]{Na78} Y. Nakagawa, {\em On the Alexander polynomials of slice links}, Osaka
J. Math. 15 (1978), no. 1, 161--182
\bibitem[Ni08]{Ni08} Y. Ni, {\em Addendum to:``Knots, sutures and excision"}, Preprint (2008)
\bibitem[Nic03]{Nic03}
L. Nicolaescu, {\em The Reidemeister torsion of 3-manifolds}, de Gruyter Studies in Mathematics, 30. Walter de Gruyter \& Co., Berlin, 2003.
\bibitem[OS04]{OS04}
P. Ozsv\'ath and Z. Szab\'o, {\em Holomorphic disks and 3-manifold invariants: properties and applications}, Annals of Mathematics 159 (2004) no. 3, 1159-1245.
\bibitem[Pa07]{Pa07}
A. Pajitnov, {\em Novikov homology, twisted Alexander polynomials, and Thurston cones},  St. Petersburg Math. J.  18  (2007),  no. 5, 809--835

\bibitem[Pa08]{Pa08}
A. Pajitnov, {\em On the tunnel number and the Morse-Novikov number of knots}, Preprint (2008)
\bibitem[Po97]{Po97}
J. Porti, {\em Torsion de Reidemeister pour les vari\'et\'es hyperboliques}, vol. 128, Mem. Amer.
Math. Soc., no. 612, AMS (1997)
\bibitem[Ri71]{Ri71}
R. Riley, {\em
Homomorphisms of knot groups on finite groups}, Math. of Computation, 25, 603�619 (1971)
\bibitem[Ri72]{Ri72}
R. Riley, {\em Parabolic representations of knot groups}, I. Proc. London Math. Soc. (3) 24 (1972),
217--242.
\bibitem[Ri84]{Ri84}
R. Riley, {\em Nonabelian representations of 2--bridge knots groups}, Quart. J. Math. Oxford Ser. (2) 35 (1984), 191--208
\bibitem[Ri90]{Ri90}
R. Riley, {\em Growth of order of homology of cyclic branched covers of knots}, Bull.
London Math. Soc. 22 (1990), 287 - 297.

\bibitem[Sa06]{Sa06}
T. Sakasai, {\em
Higher-order Alexander invariants for homology cobordisms of a surface},
Intelligence of Low Dimensional Topology 2006, Series on Knots and Everything 40, 271--278.
\bibitem[Sa07]{Sa07}
T. Sakasai, {\em Computations of noncommutative Alexander invariants for string links},
slides of talk, available at
\texttt{http://www.math.titech.ac.jp/\~\,sakasai/}
\bibitem[Sa08]{Sa08}
T. Sakasai, {\em The Magnus representation and higher-order Alexander
invariants for homology cobordisms of surfaces},
Algebraic \& Geometric Topology 8 (2008) 803--848
\bibitem[Se35]{Se35}
H. Seifert, {\em \"Uber das Geschlecht von Knoten},
Math. Ann. 110 (1935), no. 1, 571--592.
\bibitem[SW02]{SW02}
D. Silver and  S. Williams, {\em Mahler measure, links and homology
growth}, Topology 41 (2002), 979-991.
\bibitem[SW06]{SW06}
D. Silver and  S. Williams, {\em
Twisted Alexander Polynomials Detect the Unknot}, Algebraic and Geometric Topology, 6 (2006), 1893-1907.
\bibitem[SW09a]{SW09a}
D. Silver and  S. Williams, {\em Nonfibered knots and representation shifts}, Proceedings of Postnikov Memorial Conference, Banach Center Publications, 85 (2009), 101--107.
\bibitem[SW09b]{SW09b}
D. Silver and  S. Williams, {\em
Twisted Alexander Polynomials and Representation
Shifts}, Bull. London Math. Soc., 41 (2009), 535--540
\bibitem[SW09c]{SW09c}
D. Silver and  S. Williams, {\em
Dynamics of Twisted Alexander Invariants}, Topology and its Applications, 156 (2009), 2795--2811
\bibitem[SW09d]{SW09d}
D. Silver and  S. Williams, {\em Alexander-Lin twisted polynomials}, Preprint (2009)
\bibitem[SW09e]{SW09e}
D. Silver and  S. Williams, {\em
On a Theorem of Burde and de Rham}, Preprint (2009)
\bibitem[St74]{St74} {R. Strebel}, {Homological methods applied to the derived series of groups}, Comment. Math.
Helv. 49 (1974), 302--332.

\bibitem[Sug07]{Sug07}
K. Sugiyama, {\em An analog of the Iwasawa conjecture for a complete hyperbolic threefold of finite volume},
J. Reine Angew. Math. 613 (2007), 35--50.
\bibitem[Sum71]{Sum71}
D. W. Sumners, {\em Invertible knot cobordisms}, Comment. Math. Helv. 46: 240--256 (1971)
\bibitem[Suz04]{Suz04}
M. Suzuki, {\em
Twisted Alexander polynomial for the Lawrence-Krammer representation},
Bull. Aust. Math. Soc. 70, No. 1, 67-71 (2004).
\bibitem[Tam02]{Tam02}
A. Tamulis, {\em Knots of Ten or Fewer Crossings of Algebraic Order Two},  J. Knot Theory Ramifications  11  (2002),  no. 2, 211--222.
\bibitem[Ta94]{Ta94} C. H. Taubes, {\em The Seiberg-Witten invariants and symplectic forms}, Math. Res. Lett. 1: 809--822 (1994)
\bibitem[Ta95]{Ta95} C. H. Taubes, {\em More constraints on symplectic forms from Seiberg-Witten invariants}, Math. Res. Lett. 2: 9--13 (1995)
\bibitem[Th76]{Th76} W. P. Thurston,
{\em Some simple examples of symplectic manifolds}, Proc. Amer.
Math. Soc. 55 (1976), no. 2, 467--468.
\bibitem[Th82]{Th82} W. P. Thurston,
{\em Three dimensional manifolds, Kleinian groups and hyperbolic geometry}, Bull. Amer. Math.
Soc. 6 (1982)
\bibitem[Th86]{Th86} W. P. Thurston, {\em A norm for the homology of 3--manifolds}, Mem.
Amer. Math. Soc. 339: 99--130 (1986)
\bibitem[Th87]{Th87} W. P. Thurston, {\em  Geometry and Topology of 3-Manifolds}, Princeton University
Lecture Notes, 1987.
\bibitem[To53]{To53}
G. Torres, {\em On the Alexander polynomial}, Ann. of Math. 57 (1953), 57--89.
\bibitem[Tr61]{Tr61}
H. F. Trotter, {\em Periodic automorphisms of groups and knots}, Duke Math. J. 28: 553-557 (1961)
\bibitem[Tu75]{Tu75}
V. Turaev, {\em The Alexander polynomial of a three-dimensional manifold}, (Russian) Mat. Sb. (N.S.) 97(139) (1975), no. 3(7), 341--359
\bibitem[Tu86]{Tu86} V. Turaev, {\em
Reidemeister torsion in knot theory},
 Russian Math. Surveys 41 (1986), no. 1, 119--182.
 \bibitem[Tu97]{Tu97} V. Turaev, {\em Torsion invariants of ${\rm Spin}\sp c$-structures on $3$-manifolds},  Math. Res. Lett.  4  (1997),  no. 5, 679--695.
 \bibitem[Tu98]{Tu98} V. Turaev, {\em
 A combinatorial formulation for the Seiberg-Witten invariants of $3$-manifolds},  Math. Res. Lett.  5  (1998),  no. 5, 583--598.
\bibitem[Tu01]{Tu01} V. Turaev, {\em Introduction to combinatorial torsions}, Birkh\"auser, Basel, (2001)
\bibitem[Tu02a]{Tu02a}
V. Turaev, {\em Torsions of 3--manifolds}, Progress in
Mathematics, \textbf{208}. Birkhauser Verlag, Basel, 2002.
\bibitem[Tu02b]{Tu02b}
V. Turaev, {\em A homological estimate for the Thurston norm},
unpublished note (2002), arXiv:math.~GT/0207267
\bibitem[Tu02c]{Tu02c}
V. Turaev, {\em A norm for the cohomology of 2-complexes},  Algebr. Geom. Topol. 2, 137-155 (2002).
\bibitem[Vi99]{Vi99} S. Vidussi, {\em The Alexander norm is smaller than the Thurston norm; a
Seiberg--Witten proof}, Prepublication Ecole Polytechnique 6 (1999)
 \bibitem[Vi03]{Vi03} S. Vidussi,
{\em Norms on the cohomology of a 3-manifold and SW theory},
  Pacific J. Math.  208,  no. 1: 169--186 (2003)
\bibitem[Wa94]{Wa94}
M. Wada, {\em Twisted Alexander polynomial for finitely presentable groups}, Topology 33, no. 2:
241--256 (1994)
\bibitem[Wh87]{Wh87}
W. Whitten, {\em Knot complements and groups}, Topology 26: 41--44 (1987)

\end{thebibliography}
\end{document}